\documentclass[11pt]{amsart}

\usepackage[margin=1in]{geometry}
\usepackage[T1]{fontenc}
\usepackage[compress]{cite}
\usepackage[utf8]{inputenc}
\usepackage{amsmath,amssymb,amsthm,mathtools}
\usepackage{microtype,enumerate,enumitem}
\usepackage[
colorlinks=true,
linkcolor=blue,
citecolor=blue,
urlcolor=blue
]{hyperref}

\newcommand{\R}{\mathbb R}
\newcommand{\dd}{\,\mathrm d}
\newcommand{\supp}{\operatorname{supp}}
\newcommand{\diver}{\operatorname{div}}
\newcommand{\PV}{\operatorname{p.v.}}
\newcommand{\Sch}{\mathcal S}
\newcommand{\norm}[1]{\left\|#1\right\|}
\newcommand{\pair}[2]{\left\langle #1,#2\right\rangle}
\newcommand{\ip}[2]{\left\langle #1,#2\right\rangle}

\newtheorem{theorem}{Theorem}[section]
\newtheorem{lemma}[theorem]{Lemma}

\newtheorem{proposition}[theorem]{Proposition}
\newtheorem{corollary}[theorem]{Corollary}
\theoremstyle{remark}
\newtheorem{remark}[theorem]{Remark}
\numberwithin{equation}{section}

\begin{document}

\title[]{Liouville-type theorems and perturbation uniqueness for stationary incompressible MHD with directional fractional dissipation}
\author{Yong Zeng}
\address{School of Mathematics and Statistics, Chongqing Technology and Business University, Chongqing, 400067, China}
\address{Chongqing Key Laboratory of Statistical Intelligent Computing and Monitoring, Chongqing Technology and Business University}
\email{yzeng@ctbu.edu.cn}

	\maketitle
	
	\begin{abstract}
		We study Liouville-type theorems for stationary solutions of the incompressible
		magnetohydrodynamics equations in $\mathbb{R}^n$, with directional	fractional dissipation $(-\partial_i^2)^{\alpha}u$ and
		$(-\partial_j^2)^{\beta}b$, where $n\ge 2$ and $0<\alpha,\beta\le 1$.
		We prove that $u\equiv b\equiv 0$ by assuming that
		$u\in L^{p,\infty}(\mathbb{R}^n)$ and
		$b\in L^{q,\infty}(\mathbb{R}^n)$, with suitable conditions on $p$ and $q$.
		The proof uses a localized fractional energy identity and does not assume
		finiteness of the global homogeneous Sobolev energy. The method also applies to the case of full fractional dissipation.
		For the system with common-directional fractional dissipation, we prove
		stationary uniqueness for compactly supported perturbations of arbitrary smooth
		transverse shear backgrounds when $0<\alpha,\beta<1$.
		We also establish a quantitative uniqueness criterion for perturbations of
		backgrounds with bounded gradients that are not necessarily compactly supported,
		for all $0<\alpha,\beta\le 1$, assuming directional localization.
	\end{abstract}
	
		\tableofcontents
	\section{Introduction and main results}
	
	Let $n\ge 2$ and $0<\alpha,\beta\le1$. We consider the stationary incompressible
	fractional MHD system
	\begin{equation}\label{eq:system1}
		\begin{cases}
			(-\Delta)^{\alpha}u+(u\cdot\nabla)u
			-(b\cdot\nabla)b+\nabla P=0,\\
			(-\Delta)^{\beta}b+(u\cdot\nabla)b
			-(b\cdot\nabla)u=0,\\
			\diver u=\diver b=0
		\end{cases}
		\qquad\text{in }\R^n,
	\end{equation}
	and its directional counterpart
	\begin{equation}\label{eq:system2}
		\begin{cases}
			(-\partial_i^2)^{\alpha}u+(u\cdot\nabla)u
			-(b\cdot\nabla)b+\nabla P=0,\\
			(-\partial_j^2)^{\beta}b+(u\cdot\nabla)b
			-(b\cdot\nabla)u=0,\\
			\diver u=\diver b=0
		\end{cases}
		\qquad\text{in }\R^n,
	\end{equation}
	where $i,j\in\{1,\ldots,n\}$ need not coincide. Here $u=(u_1(x),u_2(x),\cdots,u_n(x))$ denotes the velocity field, $b=(b_1(x),b_2(x),\cdots,b_n(x))$ denotes the magnetic field, and $P=P(x)$ denotes the pressure.	We write
	\[
	\Lambda=(-\Delta)^{1/2},
	\qquad
	\Lambda_j=(-\partial_j^2)^{1/2},
	\quad j=1,\ldots,n.
	\]
	Thus, for $0<\gamma<1$, $\Lambda^{2\gamma}=(-\Delta)^\gamma$,
	$\Lambda_j^{2\gamma}=(-\partial_j^2)^\gamma$.
	At the classical endpoint $\gamma=1$, these operators are
	$\Lambda^2=-\Delta$, $\Lambda_j^2=-\partial_j^2$.
	For $f\in\mathcal S(\mathbb R^n)$ and $0<\gamma<1$, the
	fractional Laplacians have the singular-integral representations
	\begin{equation}\label{eq:def-Lambda}
		\Lambda^{2\gamma}f(x)
		=
		c_{n,\gamma}\,\PV
		\int_{\mathbb R^n}
		\frac{f(x)-f(x+y)}{|y|^{n+2\gamma}}\,\dd y,
	\end{equation}
	and
	\begin{equation}\label{eq:def-Lambdaj}
		\Lambda_j^{2\gamma}f(x)
		=
		c_{\gamma}\,\PV
		\int_{\mathbb R}
		\frac{f(x)-f(x+he_j)}{|h|^{1+2\gamma}}\,\dd h,
	\end{equation}
	where
	\[
	c_{n,\gamma}
	=
	\frac{4^\gamma\Gamma(n/2+\gamma)}
	{\pi^{n/2}|\Gamma(-\gamma)|},
	\qquad
	c_{\gamma}
	=
	\frac{4^\gamma\Gamma(1/2+\gamma)}
	{\sqrt{\pi}\,|\Gamma(-\gamma)|}.
	\]
	For Lebesgue or weak-Lebesgue fields $v$, the operators are
	understood distributionally by duality:
	\[
	\pair{\Lambda_j^{2\gamma}v}{\phi}
	=
	\int_{\mathbb R^n}v\cdot\Lambda_j^{2\gamma}\phi\,\dd x,
	\qquad
	\phi\in C_c^\infty(\R^n),
	\]
	and analogously for $\Lambda^{2\gamma}$. 
	
	The Liouville problem for the stationary incompressible
	Navier--Stokes equations asks which conditions force
	a smooth solution to vanish. In three dimensions, a central open
	question is whether every smooth solution satisfying
	\[
	-\Delta u+(u\cdot\nabla)u+\nabla P=0,
	\qquad
	\diver u=0,
	\]
	together with
	\[
	u(x)\longrightarrow0
	\quad\text{as }|x|\longrightarrow\infty,
	\qquad
	\int_{\R^3}|\nabla u|^2\,\dd x<\infty,
	\]
	must be trivial. A classical sufficient condition is
	$u\in L^{9/2}(\R^3)$, given by Galdi \cite{Galdi}.
	The exponent $9/2$ arises naturally from the cubic terms 
	in the localized energy identity, whereas finite Dirichlet energy	provides the larger Sobolev exponent $6$. This mismatch motivates	the search for additional conditions on integrability, oscillation,	pressure, or spatial distribution. Using the head pressure	$Q=P+\frac12|u|^2$, Chae \cite{Chae2014} proved triviality under
	$\Delta u\in L^{6/5}(\R^3)$, an assumption with the same scaling
	as finite Dirichlet energy. Chae and Wolf
	\cite{ChaeWolf2016} obtained a logarithmic improvement of
	$L^{9/2}(\R^3)$ and further criteria involving the head pressure.
	Additional developments of the pressure approach appear in
	\cite{Chae2020,Chae2025,Chae2026}. A different direction was pursued
	by Seregin \cite{Seregin2016}, who established a Liouville theorem
	under $u\in L^6(\R^3)\cap\mathrm{BMO}^{-1}(\R^3)$, and subsequently
	developed criteria involving homogeneous Morrey spaces
	\cite{Seregin2019}. Chamorro, Jarr\'in, and
	Lemari\'e-Rieusset \cite{ChamorroJarrinLemarie2021} proved further
	Lebesgue and Morrey-space results, while Jarr\'in \cite{Jarrin2020}
	studied Lorentz and Morrey formulations. Related integrability and
	large-scale criteria were obtained by Kozono, Terasawa, and
	Wakasugi \cite{KozonoTerasawaWakasugi2017}, Seregin and Wang
	\cite{SereginWang2020}, and Tsai \cite{Tsai2021}.
	Chae and Wolf \cite{ChaeWolf2019} established further stationary
	Navier--Stokes Liouville results, and Chae
	\cite{ChaeAnisotropicNS2023} considered anisotropic assumptions.
	Recently, Dong, Ma, and Zhang \cite{DMZ}  obtained Liouville-type results under conditions involving the velocity gradient. See \cite{ChaeAnisotropicNS2023} for result in anisotropic Lebesgue spaces.
	
	For the stationary incompressible magnetohydrodynamics (MHD)
	equations, the velocity and magnetic field must be considered
	together:
	\[
	\begin{cases}
		-\Delta u+(u\cdot\nabla)u-(b\cdot\nabla)b+\nabla P=0,\\
		-\Delta b+(u\cdot\nabla)b-(b\cdot\nabla)u=0,\\
		\diver u=\diver b=0.
	\end{cases}
	\]
	The divergence-free conditions yield cancellations between the
	two equations, but localization leaves mixed fluxes of the form
	$|u||b|^2$, in addition to the velocity flux $|u|^3$.
	The pressure is determined, up to a harmonic contribution, by the
	combined stress $u\otimes u-b\otimes b$. Consequently, stationary
	MHD Liouville criteria must account for the interaction between
	the integrability of the two fields. Chae and Wolf
	\cite{ChaeWolfMHD2021} studied Liouville theorems for stationary
	MHD and Hall--MHD. For the MHD case, they proved that  if there exist smooth matrix valued potential functions $\Phi,\Psi$ such that $\nabla\cdot\Phi=u$ and $\nabla\cdot \Psi=b$, whose $L^6$ mean oscillations have certain growth condition near infinity, then $u=b=0$. Chen, Li, and Wang
	\cite{ChenLiWang2022} obtained further studied these systems. They proved that $u=b=0$ provided that $u\in BMO^{-1}(\R^3)$, $b\in L^6(\R^3)$ for the MHD equations and $u\in L^6(\R^3)$, $b\in L^3(\R^3)$ for the MHD equations with partial viscosity or diffusivity. Later, Wang  and Guo \cite{WG} proved that  $u=b=0$ provided that $u \in BMO^{-1}(\mathbb{R}^3)$ and $b \in L^{p,\infty}(\mathbb{R}^3)$ with $2 < p \le 6$. Moreover, they also prove that if
	\[
	u \in L^p(\mathbb{R}^3), \ b \in L^q(\mathbb{R}^3) \quad \text{with} \quad \frac{2}{3}\le\frac{1}{p} + \frac{2}{q} <1, \quad 3 \le p \le \frac{9}{2}, \quad 2 \le q \le 6,
	\]
	then $u$ and $b$ are vanishing for the stationary 3D MHD equations with partial viscosity and partial diffusivity.
	Cho, Neustupa, and Yang \cite{ChoNeustupaYang2024} developed a
	common treatment of the stationary Navier--Stokes, MHD, and
	Hall--MHD equations under conditions on the growth of suitable
	Lebesgue norms. Chae \cite{ChaeAnisotropicMHD2023} investigated
	anisotropic Liouville criteria for MHD in $\R^n$, and Chae and Lee
	\cite{ChaeLee2024} established further results for stationary MHD
	in $\R^3$ under coupled integrability assumptions. Zhang \cite{Zhang} obtained  Liouville-type theorems by assume the Lebesgue norms of the velocity and
	magnetic field on the annulus satisfy certain growth conditions. Wang and Yang \cite{WangYang2025} proved Liouville
	theorems for the three-dimensional stationary MHD and Hall--MHD
	equations with nonzero constant vectors at infinity.
	
	Fractional dissipation changes both the natural energy space and
	the localization argument. To compare the conventions used in the
	literature, consider
	\[
	(-\Delta)^\alpha u+(u\cdot\nabla)u+\nabla P=0,
	\qquad
	\diver u=0,
	\]
	where $\alpha$ denotes the order of dissipation.
	Wang and Xiao \cite{WangXiao2018} proved that the smooth solution $u\in\dot{H}^\alpha(\R^n)\cap L^{3n/(n-1)}(\R^n)$ of the stationary fractional Navier–Stokes system is trivial for $\alpha\in (0,1)$. 
	Yang	\cite{Yang2022} proved the same result for $5/6\le \alpha<1$ for the three-dimensional case\footnote{Their theorem statement does not require \(u\in\dot H^\alpha(\mathbb R^3)\), but this condition is used in the proof.}.
	Chamorro and Poggi \cite{ChamorroPoggi2025} obtained an almost sharp
	Liouville theorem for stationary fractional Navier--Stokes
	equations under assumptions $u\in\dot{H}^\alpha(\R^3)\cap L^{\frac{6-\varepsilon}{3-2\alpha}}(\R^3)$ for $1/2<\alpha<1$ and $u\in \dot{H}^\alpha(\R^3)\cap L^{\frac{6-\varepsilon}{3-2\alpha}}(\R^3)\cap L^{\frac{6+\varepsilon}{3-2\alpha}}(\R^3)$ for $3/10<\alpha<1/2$ for small $\varepsilon$. Jarr\'in and Vergara-Hermosilla
	\cite{JarrinVergara2024} obtained the Liouville-type theorem for $u\in \dot{H}^\alpha(\R^3)\cap L^p(\R^3)$ for $\max(3/(2\alpha-1),3/\alpha)<p\le 9/2$. Wang, Yang, and Yu \cite{WangYangYu2025} studied
	Liouville theorems under assumption $\nabla u\in L^{1+1/\alpha}(\R^3)$ for a system with aymptotic condition $u\to u_\infty$ as $|x|\to\infty$. More recently, Lee and Lee
	\cite{LeeLee2026} proved results under velocity integrability or gradient integrability along with a large-scale Morrey-type condition 
	\[\limsup_{r \to \infty} r^{-\lambda} \left( \int_{A(r)} \left| \Lambda^\alpha u \right|^2 dx \right)^{\frac{1}{2}} < +\infty \quad\text{for some}\quad 0\le\lambda<\alpha.
	\]
	Eguchi and  Vergara-Hermosilla \cite{EH} establish a Liouville type result with assumptions that $u\in \dot{H}^\alpha(\R^3)$ and the radial component of the velocity field $u_\rho \in L^r(\mathbb{R}^3)$ for $3/(2\alpha) \le r \le 3/(2\alpha-1)$.

	For the fractional MHD system, the two dissipation exponents may be
	different:
	\[
	\begin{cases}
		(-\Delta)^\alpha u+(u\cdot\nabla)u
		-(b\cdot\nabla)b+\nabla P=0,\\
		(-\Delta)^\beta b+(u\cdot\nabla)b
		-(b\cdot\nabla)u=0,\\
		\diver u=\diver b=0.
	\end{cases}
	\]
	Here the energy spaces
	$\dot H^\alpha$ and $\dot H^\beta$ yield different Sobolev
	exponents, and the mixed flux links the admissible ranges for
	the two fields. 
	Zeng \cite{Zeng2025} established
	Liouville theorems for the three-dimensional stationary fractional	MHD system under anisotropic Lebesgue assumptions and finite	fractional energy. As a consequence, smooth finite-energy solutions vanish when both dissipation orders, $\alpha$ and
	$\beta$, belong to $[1/2,5/6]$. Lee and Lee \cite[Appendix~A]{LeeLee2026} extended their
	kernel and bootstrap arguments to a further unequal-order
	regime. Under assumptions
	\[\limsup_{r \to \infty} \left( r^{-\lambda_1} \left( \int_{A(r)} \left| \Lambda^\alpha u \right|^2 dx \right)^{\frac{1}{2}} + r^{-\lambda_2} \left( \int_{A(r)} \left| \Lambda^\beta B \right|^2 dx \right)^{\frac{1}{2}} \right) < +\infty,
	\]
	for some $0\le \lambda_1<\alpha$ and $0\le \lambda_2<\beta$, they obtained vanishing result when
	\[
	\frac{1}{2}\le \alpha<5/2-2\beta,\qquad \frac56<\beta<1,
	\]
	with an additional  assumption $(u,b)\in L^{\frac{6}{3-2\alpha}}(\R^3)\times L^{\frac{6}{3-2\beta}}(\R^3)$. 
	
    We note that all the above results concerning the fractional Navier–Stokes and fractional MHD equations rely on assumptions on $\Lambda^\alpha u$ or $\nabla u$. Therefore, our primary interest is in establishing Liouville-type theorems for \eqref{eq:system1} and \eqref{eq:system2} with full or directional fractional dissipation, without imposing any integrability assumption on $\Lambda^\alpha u$ and $\Lambda^\beta b$. Our first result states as follows.

	\begin{theorem}\label{thm:main}
		Let $0<\alpha,\beta\le1$. Let $(u,b,P)$ be a smooth solution of
		\eqref{eq:system1} or \eqref{eq:system2}. Suppose
		$$
		u\in L^{p,\infty}(\R^n),
		\qquad b\in L^{q,\infty}(\R^n),
		$$
		where $p$ and $q$ satisfy
		\begin{equation}\label{eq:ranges}
				3<p<
				\min\left\{\frac{3n}{n-1},\frac{2n}{n-2\alpha}\right\},
				\quad
				2<q<\frac{2n}{n-2\beta},\quad
				\frac{n-1}{n}<\frac1p+\frac2q<1.
		\end{equation}
		Then $u=b=0$ and $\nabla P=0$.
	\end{theorem}
	
	Some remarks are in order.
	
	\begin{remark}
		The assertion applies only when \eqref{eq:ranges} is nonempty.
		In particular, $p>3$ and $p<2n/(n-2\alpha)$ require
		$\alpha>n/6$. Since $\alpha\le1$, the theorem has no admissible
		exponents when $n\ge6$.
	\end{remark}
	
	\begin{remark}[Strong-space endpoints]\label{rem:endpoint}
		For strong-space assumptions $u\in L^p(\R^n)$ and
		$b\in L^q(\R^n)$, the conclusion also holds under
		\begin{equation}\label{eq:strong-range}
		3\le p\le
		\min\left\{\frac{3n}{n-1},\frac{2n}{n-2\alpha}\right\},
		\quad
		2<q\le\frac{2n}{n-2\beta},
		\quad
		\frac{n-1}{n}\le\frac{1}{p}+\frac{2}{q}\le1.
		\end{equation}
		In the estimates \eqref{eq:u-flux}--\eqref{eq:PB-flux} below, ordinary H\"older's inequality and vanishing strong-space annular norms replace the weak-space estimates, ensuring the validity of the endpoint estimates. Also, use the endpoint case of the cut-off error Lemma \ref{lem:cutoff-error} for $\alpha,\beta<1$ and \eqref{eq:classical-cutoff-end} for $\alpha=1$ or $\beta=1$.
	\end{remark}
	
	\begin{remark}[Comparison with Chen--Li--Wang \cite{ChenLiWang2022} and Wang--Guo\cite{WG}]
		Chen, Li, and Wang \cite{ChenLiWang2022} and Wang and Guo
		\cite{WG} establish Liouville theorems for classical MHD systems in
		$\mathbb R^3$. Their full-dissipation results assume
		$u\in\mathrm{BMO}^{-1}(\mathbb R^3)$ and impose Lorentz-space
		conditions on the magnetic field: Chen--Li--Wang require
		$b\in L^{6,\infty}(\mathbb R^3)$, while Wang--Guo allow
		$b\in L^{r,\infty}(\mathbb R^3)$ for $2<r\le 6$.  For the full-dissipation case, our main purpose is not to refine the results of Chen–Li–Wang and Wang–Guo, but rather to provide a different unified treatment that simultaneously yields results for both the fractional and classical cases.
		
		For partial diffusion, Chen--Li--Wang assume
		$u\in L^3(\mathbb R^3)$ and $b\in L^6(\mathbb R^3)$, with a
		positive amount of viscosity and magnetic diffusivity. Wang--Guo
		consider the system with diffusion in the same single coordinate
		direction in both equations and assume
		\[
		u\in L^p(\mathbb R^3),\quad b\in L^q(\mathbb R^3),\quad
		3\le p\le\frac92,\quad 2< q\le6,\quad
		\frac{2}{3}\le\frac1p+\frac2q\le 1.
		\]
		By contrast, the present main theorem is formulated in $\mathbb R^n$,
		with $n\ge3$, and treats full or single-directional fractional
		dissipation, with possibly different orders $\alpha,\beta$ and
		different dissipative directions. Its weak-Lorentz hypotheses are $u\in L^{p,\infty}(\mathbb R^n)$ and
		$b\in L^{q,\infty}(\mathbb R^n)$ with \eqref{eq:ranges}.
		Thus, at the classical three-dimensional endpoint
		$\alpha=\beta=1$, the main theorem gives a weak-Lorentz result
		for the strict interior range
		\[
		3<p<\frac92,\qquad 2<q<6,\qquad
		\frac23<\frac1p+\frac2q<1.
		\]
		As noted in Remark \eqref{rem:endpoint}, the admissible range \eqref{eq:strong-range} of indices in our strong-space results agrees with that obtained by Wang and Guo.
		
		The proof strategies also differ. Both Chen--Li--Wang and Wang--Guo
		derive localized Caccioppoli estimates, use potential
		representations and local Stokes estimates to control pressure,
		and apply an iteration to obtain finite Dirichlet energy before
		proving vanishing. The present proof instead localizes the
		fractional dissipation directly, estimates the cutoff error from
		the weak-Lorentz norms, and does not first establish a global
		homogeneous Sobolev or Dirichlet-energy bound. Vanishing of the
		localized dissipation then confines Fourier support to the zero set
		of the dissipation symbol; the finite weak-Lebesgue assumptions
		exclude nonzero polynomial solutions.
	\end{remark}
	
	\begin{remark}[Obstacle to a Caccioppoli--iteration argument]
		Let
		\[
		\mathcal D(r)
		:=\|\Lambda^\alpha u\|_{L^2(B_r)}^2
		+\|\Lambda^\beta b\|_{L^2(B_r)}^2 .
		\]
		The classical Caccioppoli argument relies on an estimate of the form
		\[
		\mathcal D(\rho)
		\le \theta \mathcal D(\tau)
		+\mathcal R(\rho,\tau),
		\qquad 0<\theta<1,
		\]
		where $\mathcal R$ contains only lower-order cutoff terms.  This
		allows one to apply Giaquinta's iteration lemma.
		
		For $0<\gamma<1$, however, testing a fractional dissipative term
		with $\eta^2 v$ gives
		\[
		\left\langle \Lambda^{2\gamma}v,\eta^2v\right\rangle
		=
		\|\Lambda^\gamma(\eta v)\|_2^2
		-
		c_\gamma
		\iint
		\frac{|\eta(x)-\eta(y)|^2v(x)\cdot v(y)}
		{|x-y|^{n+2\gamma}}\,\dd x\,\dd y ,
		\]
		with the analogous formula for directional fractional dissipation.
		Consequently,
		\[
		\|\Lambda^\gamma(\eta v)\|_2^2
		\le
		\left\langle \Lambda^{2\gamma}v,\eta^2v\right\rangle
		+
		C\int_{\mathbb R^n}|v(x)|^2K_{\gamma,\eta}(x)\,\dd x,
		\]
		where
		\[
		K_{\gamma,\eta}(x)
		=
		\int_{\mathbb R^n}
		\frac{|\eta(x)-\eta(y)|^2}{|x-y|^{n+2\gamma}}\,\dd y .
		\]
		The last term is genuinely nonlocal: even if $\eta=1$ on
		$B_\rho$ and $\supp\eta\subset B_\tau$, it involves values of $v$
		outside $B_\tau$.  Although it can be bounded by
		\[
		\int_{\mathbb R^n}|v|^2K_{\gamma,\eta_\tau}
		\lesssim
		\tau^{\,n-\frac{2n}{r}-2\gamma}
		\|v\|_{L^{r,\infty}}^2,
		\]
		this is a global Lorentz-norm term rather than a multiple of
		$\mathcal D(\tau)$.  Hence it cannot be absorbed into
		$\theta\mathcal D(\tau)$ in the classical iteration.
		
		In addition, the local stationary--Stokes estimate used for the
		pressure in the classical case has no direct fractional analogue
		under only weak-Lorentz assumptions.  Thus the pressure term cannot
		be converted into a lower-order term plus a fraction of
		$\mathcal D(\tau)$.  Therefore 	the local fractional Caccioppoli inequality required for the	Giaquinta iteration fails to close.  This is why the present proof uses an
		expanding-cutoff argument, a global Riesz-transform representation	of the pressure, and Fourier support localization instead.
	\end{remark}

	\bigskip
	
	Without global integrability or decay, the directional system
	admits nonzero shear solutions. Consider the common-direction
	system
	\begin{equation}\label{eq:MHD}
		\begin{cases}
			\Lambda_1^{2\alpha}u+(u\cdot\nabla)u+\nabla P
			=(b\cdot\nabla)b,\\
			\Lambda_1^{2\beta}b+(u\cdot\nabla)b
			=(b\cdot\nabla)u,\\
			\diver u=\diver b=0.
		\end{cases}
	\end{equation}
	Writing $x'=(x_2,\ldots,x_n)$, any smooth profiles $U(x')$ and
	$B(x')$ give a background solution
	\begin{equation}\label{eq:shear}
		u_0=U(x')e_1,\qquad b_0=B(x')e_1
	\end{equation}
	with constant pressure, where $e_1$ denotes the unit vector in $x_1$ direction. A nonzero such field has no global
	$L^{r,\infty}$ integrability for finite $r$.
	
	Set $u=u_0+v$ and $b=b_0+w$. The perturbation equations are
	\begin{equation}\label{eq:perturbation}
		\begin{cases}
			&\Lambda_1^{2\alpha}v+(u_0\cdot\nabla)v
			+(v\cdot\nabla)u_0+(v\cdot\nabla)v+\nabla P
			=(b_0\cdot\nabla)w+(w\cdot\nabla)b_0+(w\cdot\nabla)w,\\
			&\Lambda_1^{2\beta}w+(u_0\cdot\nabla)w
			+(v\cdot\nabla)b_0+(v\cdot\nabla)w
			=(b_0\cdot\nabla)v+(w\cdot\nabla)u_0+(w\cdot\nabla)v,\\
			&\diver v=\diver w=0.
		\end{cases}
	\end{equation}
	We have the following uniqueness result for compactly-supported perturbations.
	\begin{theorem}\label{thm:stability}
		Let $0<\alpha,\beta<1$, $U,B\in C^\infty(\R^{n-1})$ and let
		$(v,w,P)$ be a smooth solution of \eqref{eq:perturbation}, with
		$v,w\in C_c^\infty(\R^n)$.
		Then $v=w=0$ and $\nabla P=0$.
	\end{theorem}
	
	\begin{remark}
		Theorem~\ref{thm:stability} is a stationary uniqueness result
		for compactly supported perturbations of \emph{arbitrary}
		smooth transverse backgrounds. The proof relies on the
		nonlocal unique-continuation property of the fractional
		Laplacian (Lemma~\ref{lem:compact-support} below), which
		holds only for $0<\gamma<1$ and fails at $\gamma=1$ where
		$-\partial_1^2$ is a local operator. 
	\end{remark}
	
	\bigskip
	
	We now consider not necessarily compactly-supported perturbations of \eqref{eq:shear}. Before stating our result, we introduce the one-dimensional
	directional concentration function associated with $f$. For
	$0\ne f\in L^2(\mathbb R^n)$ and $\ell>0$, set
	$$
	\mathcal C_{f}(\ell)
	:=
	\sup_{a\in\mathbb R}
	\frac{
		\displaystyle
		\int_{\{a-\ell/2<x_1<a+\ell/2\}}\vert{}f(x)\vert{}^2\,\dd x
	}{
		\|f\|_{L^2(\mathbb R^n)}^2
	}.
	$$For $f=0$, we define $\mathcal{C}_0(\ell)=1$. This measures the largest fraction of the $L^2$-mass of $f$
	contained in a slab of width $\ell$ orthogonal to the $x_1$-axis.
	Our result reads as follows.
	
	\begin{theorem}
		\label{thm:intro-noncompact}
		Let $0<\alpha,\beta\le1$, $U,B\in W^{1,\infty}(\R^{n-1})$ and let
		$(v,w,P)$ be a smooth solution of the perturbation system
		\eqref{eq:perturbation} satisfying
		$$
		v\in L^2(\mathbb R^n)\cap L^{p,\infty}(\mathbb R^n),\quad
		w\in L^2(\mathbb R^n)\cap L^{q,\infty}(\mathbb R^n),
		$$
		where the  exponents satisfy
	\begin{equation}\label{eq:noncompact-exponents}
				3< p<\infty,\qquad
				\frac1p+\frac2q< 1.
			\end{equation}
		Suppose that, for some $\ell_v,\ell_w>0$ and
		$\kappa_v,\kappa_w\in(0,1)$,
		\begin{equation}\label{eq:concentration}
			\mathcal C_v(\ell_v)\ge\kappa_v,
			\qquad
			\mathcal C_w(\ell_w)\ge\kappa_w.
		\end{equation}
		Set
		$$
		a=C_{\alpha,\kappa_v}\ell_v^{2\alpha},
		\qquad
		d=C_{\beta,\kappa_w}\ell_w^{2\beta},\quad
		M_U=\|\nabla U\|_{L^\infty(\R^{n-1})},\quad
		M_B=\|\nabla B\|_{L^\infty(\R^{n-1})},
		$$
		where
		$$
		C_{s,\kappa}
		=
		\frac4\kappa
		\left(\frac4{\pi\kappa}\right)^{2s}.
		$$
		If
		\begin{equation}\label{eq:noncompact}
			M_Ua<2,\qquad M_Ud<2,\qquad
			M_B^2ad<(2-M_Ua)(2-M_Ud),
		\end{equation}
		then $v=w=\nabla P=0$.
	\end{theorem}
	\begin{remark}
		$(M_U+M_B)\max\{a,d\}<2$
		is a convenient sufficient condition for \eqref{eq:noncompact}.
	\end{remark}
	
	
	\begin{remark}\label{rem:intro-localized-examples}
		Many functions that are not compactly supported satisfy the directional
		concentration hypothesis \eqref{eq:concentration}. For example, let
		$g\in L^2(\mathbb R^{n-1})$, $g\not\equiv 0$, and define
		$$
		f_1(x_1,x')=e^{-\mu\vert{}x_1-a\vert{}}g(x'),\qquad
		f_2(x_1,x')=e^{-\lambda(x_1-a)^2}g(x'),
		$$
		where $\mu,\lambda>0$. For each $0<\kappa<1$, each profile satisfies
		$\mathcal C_{f_i}(\ell)\ge\kappa$ for all sufficiently large $\ell$.
		
		The algebraically localized smooth profile
		$$
		f_3(x_1,x')=(1+(x_1-a)^2)^{-m/2}g(x')
		$$
		belongs to $L^2(\mathbb R^n)$ when $m>1/2$. It has finite directional
		second moment when $m>3/2$. Even when $1/2<m\le3/2$, its $L^2$ mass
		can still be captured by a sufficiently long interval for every fixed
		$\kappa<1$.
		
		If, in addition, $g\in L^2(\mathbb R^{n-1})\cap L^r(\mathbb R^{n-1})$
		for some $1\le r<\infty$, then $f_1,f_2\in L^2(\mathbb R^n)\cap
		L^r(\mathbb R^n)$, and $f_3\in L^2(\mathbb R^n)\cap L^r(\mathbb R^n)$
		provided $m>1/2$ and $m>1/r$. For $s>0$,
		$\Lambda_1^s f_1,\Lambda_1^s f_2,\Lambda_1^s f_3\in L^2(\mathbb R^n)$
		provided, respectively, $s<3/2$, any $s>0$, and $m>1/2$.
	\end{remark}
	
	\begin{remark}\label{rem:concentration-attainment}
		For nonzero $f\in L^2$, $\mathcal C_f(\ell)\ge\kappa$ is
		equivalent to the existence of an interval
		$I$ of length $\ell$ satisfying 
		$$
		\int_{\{x_1\in I\}}\vert{}f(x)\vert{}^2\,\dd x
		\ge\kappa\|f\|_2^2.
		$$
		Indeed, by Tonelli's theorem, the function
		$$
		m_f(t):=\int_{\R^{n-1}} \vert{}f(t,x')\vert{}^2\dd x'
		$$
		is nonnegative and belongs to $L^1(\R)$, with
		$$
		\int_{\R}m_f(t)\dd t=\|f\|_{L^2(\R^n)}^2>0.
		$$
		For $a\in\R$, put
		$$
		I_a=[a-\ell/2,a+\ell/2],
		\qquad
		F_\ell(a):=\int_{I_a}m_f(t)\dd t.
		$$
		Then $F_\ell$ is continuous. Moreover,
		$$
		0\le F_\ell(a)
		\le \int_{\{\vert{}t\vert{}\ge \vert{}a\vert{}-\ell/2\}}m_f(t)\dd t
		\longrightarrow 0
		\qquad\text{as }\vert{}a\vert{}\to\infty,
		$$
		because $m_f\in L^1(\R)$.
%
		By continuity, $F_\ell$ attains its maximum on some compact
		interval $[-A,A]$. This maximum is also its global maximum,
		so there exists $a_*\in\R$ satisfying $	F_\ell(a_*)=\sup_{a\in\R}F_\ell(a)$.
		If $\mathcal C_f(\ell)\ge\kappa$, the maximizing interval
		$I_{a_*}$ satisfies
		$$
		\int_{\{x_1\in I_{a_*}\}}\vert{}f(x)\vert{}^2\dd x
		=F_\ell(a_*)
		\ge\kappa\|f\|_{L^2(\R^n)}^2.
		$$
		Conversely, any interval of length $\ell$ is, up to its
		endpoints, an interval $I_a$. Endpoints do not affect these
		integrals. Hence the existence of such an interval implies
		$\mathcal C_f(\ell)\ge\kappa$.
	\end{remark}
	
	\begin{remark}
		If we assume the strong-space condition $v\in L^2(\R^n)\cap L^p(\R^n)$ and $w\in L^2(\R^n)\cap L^q(\R^n)$, then Theorem \ref{thm:intro-noncompact} holds for $3\le p<\infty$, $\frac1p+\frac2q\le 1$. The reasons are similar to those stated in Remark \ref{rem:endpoint}. Indeed, define $q^*=2p/(p-1)$, since $1/p+2/q\le 1$, we have $q\ge q^*>2$. Therefore, $w\in L^2\cap L^q$ gives $w\in L^{q^*}$. Also, $v\in L^2\cap L^p$ for $p\ge 3$ gives $v\in L^3$. Thus, all the estimates \eqref{eq:u-weak}--\eqref{eq:Pw-weak} remain valid after replacing $(r,s)$ by $(p,q^*)$.
	\end{remark}
	
	The concentration assumption \eqref{eq:concentration} can be replaced by a directional moment condition.
	\begin{corollary}
		\label{cor:directional-moment}
		Let $0<\alpha,\beta\le1$, $U,B\in W^{1,\infty}(\mathbb R^{n-1})$ and let
		$(v,w,P)$ be a smooth solution of the perturbation system
		\eqref{eq:perturbation} satisfying
		$$
		v\in L^2(\mathbb R^n)\cap L^{p,\infty}(\mathbb R^n),\quad
		w\in L^2(\mathbb R^n)\cap L^{q,\infty}(\mathbb R^n),
		$$
		with $p> 3$ and $1/p+2/q< 1$.
		Suppose there exist $a_v,a_w\in\R$ and
		$L_v,L_w>0$ such that
		$$
		\|(x_1-a_v)v\|_2\le L_v\|v\|_2,
		\qquad
		\|(x_1-a_w)w\|_2\le L_w\|w\|_2.
		$$
		If the smallness condition in Theorem~\ref{thm:intro-noncompact} holds
		with
		$$
		\kappa_v=\kappa_w=\frac34,
		\qquad
		\ell_v=4L_v,\qquad \ell_w=4L_w,
		$$
		then $v=w=0$ and $\nabla P=0$.
	\end{corollary}
	
	\begin{proof}
		Chebyshev's inequality gives
		$$
		\int_{\{\vert{}x_1-a_v\vert{}>2L_v\}}\vert{}v\vert{}^2\,\dd x
		\le\frac{\|(x_1-a_v)v\|_2^2}{4L_v^2}
		\le\frac14\|v\|_2^2.
		$$
		Hence at least three quarters of the mass lies in
		$[a_v-2L_v,a_v+2L_v]$, whose length is $4L_v$.
		The same argument applies to $w$. The proof then follows by applying Theorem~\ref{thm:intro-noncompact}.
	\end{proof}
	
	\section{Preliminaries}
	
	\subsection{Lorentz spaces and inequalities}
	
For a measurable function $f$, define
$$
\mu_f(t)=\left|\{x\in\mathbb{R}^n:|f(x)|>t\}\right|,
$$
where $|\cdot|$ denotes the Lebesgue measure.
For $1\le p<\infty$ and $1\le q\le\infty$, a function $f$ is said to belong to the Lorentz space $L^{p,q}(\mathbb{R}^n)$ if
$\|f\|_{L^{p,q}(\mathbb{R}^n)}<\infty$, where the Lorentz quasi-norm is defined by
$$
\|f\|_{L^{p,q}}=
\begin{cases}
	\displaystyle
	\left(\int_0^\infty
	t^{q-1}\mu_f(t)^{q/p}\,\mathrm{d}t\right)^{1/q},
	&q<\infty,\\[1ex]
	\displaystyle\sup_{t>0}t\,\mu_f(t)^{1/p},
	&q=\infty.
\end{cases}
$$
	In particular, $L^{p,p}=L^p$ with equivalent norms.
	
	We use the following standard forms of H\"older's inequality  \cite{ON}
	and boundedness of Riesz transform in Lorentz spaces \cite{Stein,CF}.
	
	\begin{lemma}\label{lem:holder}
		Let $1<p_k<\infty$, $1\le q_k\le\infty$, and
		$1\le p<\infty$, $1\le q\le\infty$. If
		$$
		\frac1p=\sum_{k=1}^m\frac1{p_k},
		\qquad
		\frac1q=\sum_{k=1}^m\frac1{q_k},
		$$
		then
		$$
		\left\|\prod_{k=1}^m f_k\right\|_{L^{p,q}}
		\le C\prod_{k=1}^m\|f_k\|_{L^{p_k,q_k}}.
		$$
		In particular, if $1<a<\infty$, then
		$$
		\int_{\R^n}\vert{}fg\vert{}\,\dd x
		\le C_a\|f\|_{L^{a,\infty}}\|g\|_{L^{a',1}}.
		$$
	\end{lemma}
	
	
	
	\begin{lemma}\label{lem:riesz}
	Let $\mathcal R_m$ be the $m$-th Riesz transform, defined by
	$\widehat{\mathcal R_m f}(\xi)
	=-i\xi_m\widehat f(\xi)/|\xi|$
	for $\xi\neq0$. For $1<p<\infty$ and $1\le q\le\infty$,
	there exists $C_{p,q}>0$ such that
	\[
	\|\mathcal R_m f\|_{L^{p,q}}
	\le C_{p,q}\|f\|_{L^{p,q}}
	\qquad\text{for all }f\in L^{p,q}(\mathbb R^n).
	\]
	\end{lemma}

	\subsection{One-dimensional fractional Poincar\'e inequality}
	
	The following is a standard fractional Poincar\'e
	inequality; see, for example, \cite[Theorem 6.5]{DNPV}. 
	We record an explicit admissible constant for the
	quantitative criteria below.
	
	\begin{lemma}\label{lem:stab-poincare}
		Let $0<s<1$. If $f\in C_c^\infty(\R)$ is supported
		in an interval of length $\ell>0$, then
		\begin{equation}\label{eq:stab-poincare-1d}
			\|f\|_2^2
			\le C_s\ell^{2s}\|\Lambda_t^s f\|_2^2,
			\qquad C_s=\frac{3^{1+2s}}{c_s}.
		\end{equation}
		The same inequality holds for vector-valued $f$.
	\end{lemma}
	
\begin{proof}
	After translating the independent variable, we may assume that $
	\operatorname{supp} f\subset [0,\ell]$. Restrict the integral to the two rectangles $[0,\ell]\times[2\ell,3\ell]$
	and	$[2\ell,3\ell]\times[0,\ell]$.
	Since \(f=0\) on \([2\ell,3\ell]\), this gives
	\[
		\|\Lambda_t^s f\|_2^2
		\ge \frac{c_s}{2}
		\bigg(
		\int_0^\ell\int_{2\ell}^{3\ell}
		\frac{|f(t)|^2}{|t-r|^{1+2s}}
		\,dr\,dt 
		+\int_{2\ell}^{3\ell}\int_0^\ell
		\frac{|f(r)|^2}{|t-r|^{1+2s}}
		\,dr\,dt\bigg)=c_s
		\int_0^\ell\int_{2\ell}^{3\ell}
		\frac{|f(t)|^2}{|t-r|^{1+2s}}
		\,dr\,dt.
		.
	\]
	For \(0\le t\le\ell\) and \(2\ell\le r\le3\ell\), we have $
	|t-r|\le 3\ell$. Therefore,
	\[	\|\Lambda_t^s f\|_2^2
		\ge
		c_s3^{-1-2s}\ell^{-2s}\|f\|_2^2.
	\]
	This gives \eqref{eq:stab-poincare-1d}.
\end{proof}
	
	\subsection{Fractional compact-support property}
	
	\begin{lemma}\label{lem:compact-support}
		Let $0<s<1$. If $f\in C_c^\infty(\R)$ and
		$(-\partial_t^2)^sf\in C_c^\infty(\R)$, then $f=0$.
	\end{lemma}
	
	Lemma~\ref{lem:compact-support} is a special case of the strong
	uniqueness property established in \cite[Theorem~1.2]{GSU}:
	if $0<s<1$, $u\in H^{-r}(\R^n)$ for some $r\in\R$,
	and both $u$ and $(-\Delta)^su$ vanish on a nonempty open
	set, then $u\equiv0$.
	Indeed, under the assumptions of Lemma~\ref{lem:compact-support},
	both $f$ and $\Lambda^{2s}f$ vanish outside a sufficiently
	large interval.
	The proof in \cite{GSU} combines Carleman estimates with the
	Caffarelli--Silvestre extension.
	For independent interest, we provide a direct elementary proof of
	Lemma~\ref{lem:compact-support} in
	Appendix~\ref{compact-support}.

	\subsection{Equivalent definition of Fractional operators}
	
	We use the standard equivalence of the Fourier,
	singular integral, and symmetric-difference definitions
	of the fractional Laplacian; see \cite{Kwa}.
	For $0<\gamma<1$ and $f\in\Sch(\R^n)$,
	the following definitions are equivalent to the singular integral definitions given by \eqref{eq:def-Lambda} and \eqref{eq:def-Lambdaj}.
	\begin{enumerate}[label=\textup{(\arabic*)}]
		\item $\Lambda_j^\gamma f=\mathcal{F}^{-1}[|\xi_j|^\gamma \hat{f}(\xi)]$, \quad $\Lambda^\gamma f=\mathcal{F}^{-1}[|\xi|^\gamma \hat{f}(\xi)]$.
		\item 	\begin{align}
			\Lambda_j^{2\gamma}f(x)
			&=\frac{c_\gamma}{2}\int_\R
			\frac{2f(x)-f(x+he_j)-f(x-he_j)}
			{\vert{}h\vert{}^{1+2\gamma}}\,\dd h,
			\label{eq:symmetric-directional}\\
			\Lambda^{2\gamma}f(x)
			&=\frac{c_{n,\gamma}}2\int_{\R^n}
			\frac{2f(x)-f(x+y)-f(x-y)}
			{\vert{}y\vert{}^{n+2\gamma}}\,\dd y,
			\label{eq:symmetric-isotropic}
		\end{align}
	\end{enumerate}
	where $e_j$ denotes the unit vector in $x_j$-direction, $j=1,2,\cdots,n$.
	For Lebesgue or weak-Lebesgue fields in the main theorem, fractional	operators are understood distributionally by duality:
	$$
	\pair{\Lambda_j^{2\gamma}v}{\phi}
	=\int_{\R^n}v\cdot\Lambda_j^{2\gamma}\phi\,\dd x,
	\qquad \phi\in C_c^\infty(\R^n;\R^n),
	$$
	and similarly for the full operator.
	The pairing is finite by H\"older's inequality. Indeed,
	the fractional image of a compactly supported smooth test
	function is bounded and has, respectively, a directional
	tail of order $\vert{}x_j\vert{}^{-1-2\gamma}$ with bounded transverse
	support, or an isotropic tail of order $\vert{}x\vert{}^{-n-2\gamma}$.
	These images belong to $L^{a,1}$ for every $a>1$.
	The same duality defines tempered distributions on
	Schwartz test functions.
	
	When $\gamma=1$, the operators are the usual Laplacian
	$-\Delta$ or directional Laplacian $-\partial_j^2$, and
	no fractional machinery is needed: the energy identities
	reduce to ordinary integration by parts.
	
	The standard fractional energy representation
	\cite{DNPV} will also be used:
	$$
	\|\Lambda^\gamma f\|_2^2
	=\frac{c_{n,\gamma}}2
	\iint_{\R^n\times\R^n}
	\frac{\vert{}f(x)-f(y)\vert{}^2}{\vert{}x-y\vert{}^{n+2\gamma}}
	\,\dd y\,\dd x,
	$$
	with its one-dimensional, slice-wise counterpart for
	$\Lambda_j^\gamma$.
	
	\subsection{Localized identities and cutoff errors}
	
	Fractional product formulas and localization arguments are
	standard; see, for example,
	\cite{BWZ,ROS}.
	Those results do not directly provide the weak-Lorentz,
	directional cutoff estimates needed here.
	We therefore prove the directional formulation and give
	the isotropic modifications explicitly.
	
We define
\[
T_{j,\gamma,\varepsilon}f(x)
=c_\gamma\int_{|h|>\varepsilon}
\frac{f(x)-f(x+h e_j)}{|h|^{1+2\gamma}}\,\mathrm{d}h,\quad
T_{\gamma,\varepsilon}f(x)
=c_{n,\gamma}\int_{|y|>\varepsilon}
\frac{f(x)-f(x+y)}{|y|^{n+2\gamma}}\,\mathrm{d}y.
\]
They have the following symmetric forms, respectively,
\begin{align*}
	T_{j,\gamma,\varepsilon}f(x)
	&=\frac{c_\gamma}{2}\int_{|h|>\varepsilon}
	\frac{2f(x)-f(x+h e_j)-f(x-h e_j)}{|h|^{1+2\gamma}}\,\mathrm{d}h,\\
	T_{\gamma,\varepsilon}f(x)
	&=\frac{c_{n,\gamma}}{2}\int_{|y|>\varepsilon}
	\frac{2f(x)-f(x+y)-f(x-y)}{|y|^{n+2\gamma}}\,\mathrm{d}y.
\end{align*}
	\begin{lemma}\label{lem:truncation}
		Let $f\in C_c^\infty(\R^n)$ and $0<\gamma<1$.
		For $0<\varepsilon\le1$, both
		$$
		\Lambda_j^{2\gamma}f-T_{j,\gamma,\varepsilon}f,
		\qquad
		\Lambda^{2\gamma}f-T_{\gamma,\varepsilon}f
		$$
		are supported in a fixed bounded set.  Moreover, there is a constant $C_f$, independent of
		$\varepsilon$, such that
		\[
		\left\|\Lambda_j^{2\gamma}f-T_{j,\gamma,\varepsilon}f\right\|_{L^\infty}
		+
		\left\|\Lambda^{2\gamma}f-T_{\gamma,\varepsilon}f\right\|_{L^\infty}
		\le C_f\varepsilon^{2-2\gamma}.
		\]
		Consequently, they tend to zero in $L^{a,d}$ for every
		$1\le a<\infty$, $1\le d\le\infty$  uniformly as $\varepsilon\to0$.
	\end{lemma}
	
	\begin{proof}
		For the directional operator, symmetric form gives
		$$
		\Lambda_j^{2\gamma}f-T_{j,\gamma,\varepsilon}f
		=\frac{c_\gamma}2\int_{\vert{}h\vert{}\le\varepsilon}
		\frac{2f(x)-f(x+he_j)-f(x-he_j)}
		{\vert{}h\vert{}^{1+2\gamma}}\,\dd h.
		$$
		Taylor's theorem bounds the numerator by
		$\|\partial_{jj}f\|_\infty\vert{}h\vert{}^2$. Hence
		$$
		\|\Lambda_j^{2\gamma}f-T_{j,\gamma,\varepsilon}f\|_\infty
		\le C_f\int_0^\varepsilon h^{1-2\gamma}\,\dd h
		\le C_f\varepsilon^{2-2\gamma}.
		$$
		If $K=\supp f$, the difference vanishes outside
		$K+\overline{B_1}$. Lemma \ref{lem:holder} then gives
		\[\|\Lambda_j^{2\gamma}f-T_{j,\gamma,\varepsilon}f\|_{L^{a,d}(K+\bar{B_1})}\leq C_{a,d,f}|K+\bar{B}_1|^{\frac{1}{a}}\varepsilon^{2-2\gamma}.
		\le 
		\]
		
		For the isotropic operator, use
		\eqref{eq:symmetric-isotropic} and
		$\vert{}2f(x)-f(x+y)-f(x-y)\vert{}\le C_f\vert{}y\vert{}^2$.
		Polar integration gives the same bound
		$C_f\int_0^\varepsilon\rho^{1-2\gamma}\,\dd\rho$,
		with the same support and convergence conclusions.
	\end{proof}
	
	For real $\eta\in C_c^\infty(\R^n)$, define
	\begin{align*}
		K_{\gamma,j,\eta}(x)
		&=\int_\R
		\frac{\vert{}\eta(x)-\eta(x+he_j)\vert{}^2}
		{\vert{}h\vert{}^{1+2\gamma}}\,\dd h,\\
		K_{\gamma,\eta}(x)
		&=\int_{\R^n}
		\frac{\vert{}\eta(x)-\eta(x+y)\vert{}^2}
		{\vert{}y\vert{}^{n+2\gamma}}\,\dd y.
	\end{align*}
	
	\begin{lemma}\label{lem:localized}
		Let $0<\gamma<1$, $2<r<\infty$, and
		$v\in C^\infty(\R^n)\cap L^{r,\infty}(\R^n)$.	
		If $\int_{\R^n}\vert{}v\vert{}^2K_{\gamma,j,\eta}<\infty$, then
		\begin{equation}\label{eq:localized-identity}
				\pair{\Lambda_j^{2\gamma}v}{\eta^2v}
				=\|\Lambda_j^\gamma(\eta v)\|_2^2-\frac{c_\gamma}2
				\int_{\R^n}\int_\R
				\frac{\vert{}\eta(x)-\eta(x+he_j)\vert{}^2
					v(x)\cdot v(x+he_j)}
				{\vert{}h\vert{}^{1+2\gamma}}\,\dd h\,\dd x.
		\end{equation}
		In particular,
		\begin{equation}\label{eq:localized-inequality}
			\|\Lambda_j^\gamma(\eta v)\|_2^2
			\le\pair{\Lambda_j^{2\gamma}v}{\eta^2v}
			+\frac{c_\gamma}2\int_{\R^n}\vert{}v\vert{}^2K_{\gamma,j,\eta}\,\dd x.
		\end{equation}
		For the full operator, if
		$\int_{\R^n}\vert{}v\vert{}^2K_{\gamma,\eta}<\infty$, then
		\begin{equation}\label{eq:localized-identity2}
				\pair{\Lambda^{2\gamma}v}{\eta^2v}
				=\|\Lambda^\gamma(\eta v)\|_2^2-\frac{c_{n,\gamma}}2
				\int_{\R^n}\int_{\R^n}
				\frac{\vert{}\eta(x)-\eta(x+y)\vert{}^2v(x)\cdot v(x+y)}
				{\vert{}y\vert{}^{n+2\gamma}}\,\dd y\,\dd x,
		\end{equation}
		and
		\begin{equation}\label{eq:localized-inequality2}
			\|\Lambda^\gamma(\eta v)\|_2^2
			\le\pair{\Lambda^{2\gamma}v}{\eta^2v}
			+\frac{c_{n,\gamma}}2\int_{\R^n}\vert{}v\vert{}^2K_{\gamma,\eta}\,\dd x.
		\end{equation}
	\end{lemma}
	
	\begin{proof}
		We prove the directional identity. Symmetrize the truncated
		kernel and use
		$$
		(a-b)\cdot(s^2a-t^2b)
		=\vert{}sa-tb\vert{}^2-a\cdot b\,(s-t)^2.
		$$
		This yields
		$$
			\int\eta^2v\cdot T_{j,\gamma,\varepsilon}v
			=\int\eta v\cdot T_{j,\gamma,\varepsilon}(\eta v)-\frac{c_\gamma}2
			\int_{\R^n}\int_{\vert{}h\vert{}>\varepsilon}
			\frac{\vert{}\eta(x)-\eta(x+he_j)\vert{}^2v(x)\cdot v(x+he_j)}
			{\vert{}h\vert{}^{1+2\gamma}}\,\dd h\,\dd x.
		$$
		Since $\eta v\in C_c^\infty(\R^n)$, Lemma~\ref{lem:truncation}
		and Plancherel's theorem give
		$$
		\int_{\R^n}\eta v\cdot T_{j,\gamma,\varepsilon}(\eta v)
		\longrightarrow\|\Lambda_j^\gamma(\eta v)\|_2^2.
		$$
		The absolute integral of the cross term is bounded by
		$\int_{\R^n}\vert{}v\vert{}^2K_{\gamma,j,\eta}$, using
		$$
		\vert{}v(x)\cdot v(x+he_j)\vert{}
		\le\frac12\bigl(\vert{}v(x)\vert{}^2+\vert{}v(x+he_j)\vert{}^2\bigr)
		$$
		and translation symmetry. Thus dominated convergence
		applies to that term.
		
		Finally,
		$$
		\int_{\R^n}\eta^2v\cdot T_{j,\gamma,\varepsilon}v\,\dd x
		=\int_{\R^n} v\cdot T_{j,\gamma,\varepsilon}(\eta^2v)\,\dd x.
		$$
		Lemma~\ref{lem:truncation} gives convergence of the test
		function images in $L^{r',1}$, and Lorentz H\"older
		therefore gives convergence to
		$\pair{\Lambda_j^{2\gamma}v}{\eta^2v}$.
		This proves \eqref{eq:localized-identity}; the cross-term
		bound proves \eqref{eq:localized-inequality}.
		
		For the full operator, replace $h\in\R$ by $y\in\R^n$,
		$\vert{}h\vert{}^{-1-2\gamma}$ by $\vert{}y\vert{}^{-n-2\gamma}$, and
		$c_\gamma$ by $c_{n,\gamma}$. The same argument proves
		\eqref{eq:localized-identity2} and
		\eqref{eq:localized-inequality2}.
	\end{proof}
	
	\begin{lemma}\label{lem:decay_cutoff_kernel}
		Let $0<\gamma<1$, and choose $A>1$ so that
		$\supp\eta\subset B_A$. Writing $x=(x_j,x'_j)$,
		$$
		K_{\gamma,j,\eta}(x)
		\le C_{\gamma,\eta}
		\mathbf1_{\{\vert{}x'_j\vert{}\le A\}}(1+\vert{}x_j\vert{})^{-1-2\gamma},
		$$
		and
		\begin{equation}\label{eq:isotropic-kernel-bound}
			K_{\gamma,\eta}(x)
			\le C_{\gamma,\eta}(1+\vert{}x\vert{})^{-n-2\gamma}.
		\end{equation}
		Both kernels belong to $L^{a,1}(\R^n)$ for every $a>1$. Here ${\bf 1}_E$ denotes the indicator function on measurable set $E$.
	\end{lemma}
	
	\begin{proof}
		For the directional kernel, the estimate
		$$
		\vert{}\eta(x+he_j)-\eta(x)\vert{}
		\le\min\{\|\partial_j\eta\|_\infty\vert{}h\vert{},
		2\|\eta\|_\infty\}
		$$
		shows that $K_{\gamma,j,\eta}$ is uniformly bounded.
		
		If $\vert{}x'_j\vert{}>A$, the kernel vanishes, since translation
		along $e_j$ preserves the transverse coordinates.
		If $\vert{}x_j\vert{}>2A$, then $\eta(x)=0$, and a nonzero
		$\eta(x+he_j)$ requires $\vert{}x_j+h\vert{}\le A$. Hence
		$$
		K_{\gamma,j,\eta}(x)
		\le\|\eta\|_\infty^2
		\int_{\{\vert{}x_j+h\vert{}\le A\}}\vert{}h\vert{}^{-1-2\gamma}\,\dd h
		\le C_{\gamma,\eta}(1+\vert{}x_j\vert{})^{-1-2\gamma}.
		$$
		Together with boundedness, this proves the claimed
		directional bound.
		
		For
		$$
		F(x)=\mathbf1_{\{\vert{}x'_j\vert{}\le A\}}
		(1+\vert{}x_j\vert{})^{-1-2\gamma},
		$$
		its distribution function satisfies
		$$
		\vert{}\{F>\lambda\}\vert{}
		\le C_A\lambda^{-1/(1+2\gamma)},
		\qquad 0<\lambda<1.
		$$
		Thus, for every $a>1$,
		$$
		\|F\|_{L^{a,1}}
		\le C\int_0^1\lambda^{-1/[a(1+2\gamma)]}\,\dd\lambda
		<\infty.
		$$
		
		For the isotropic kernel, the same Lipschitz estimate
		and polar integration give uniform boundedness.
		If $\vert{}x\vert{}>2A$, then $\eta(x)=0$, while
		$\eta(x+y)\ne0$ requires $x+y\in B_A$ and
		$\vert{}y\vert{}\ge\vert{}x\vert{}/2$. Therefore
		$$
		K_{\gamma,\eta}(x)
		\le C\int_{-x+B_A}\vert{}y\vert{}^{-n-2\gamma}\,\dd y
		\le C_{\gamma,\eta}(1+\vert{}x\vert{})^{-n-2\gamma}.
		$$
		This proves \eqref{eq:isotropic-kernel-bound}.
		The majorant has distribution function bounded by
		$C\lambda^{-n/(n+2\gamma)}$ for $0<\lambda<1$,
		which gives its $L^{a,1}$ integrability for every $a>1$.
	\end{proof}
	
	Choose $\eta\in C_c^\infty(\R^n)$ with
	$$
	0\le\eta\le1,\qquad
	\eta=1\text{ on }B_1,\qquad
	\supp\eta\subset B_2,
	$$
	and put
	$$
	\eta_R(x)=\eta(x/R),\qquad A_R=B_{2R}\setminus B_R.
	$$
	This cut-off function will be used throughout this paper. In what follows, unless otherwise specified, all limits are taken as $R\to\infty$.
	\begin{lemma}\label{lem:cutoff-error}
		Let $0<\gamma<1$, $2<r<2n/(n-2\gamma)$, and
		$v\in L^{r,\infty}(\R^n)$. Then
		\begin{equation}\label{eq:2.9}
		\int_{\R^n}\vert{}v\vert{}^2K_{\gamma,j,\eta_R}\,\dd x
		\le CR^{n-2n/r-2\gamma}\|v\|_{L^{r,\infty}}^2
		\longrightarrow0.
		\end{equation}
		The same estimate and limit hold with
		$K_{\gamma,j,\eta_R}$ replaced by $K_{\gamma,\eta_R}$. The result remains true for $r=2n/(n-2\gamma)$ with $v\in L^r(\R^n)$.
	\end{lemma}
	
	\begin{proof}
		Set $a=r/(r-2)>1$. Lorentz H\"older gives
		\begin{equation}\label{eq:lorentz-pairing}
			\int_{\R^n}\vert{}v\vert{}^2K_{\gamma,j,\eta_R}\,\dd x
			\le C\|v\|_{L^{r,\infty}}^2\|K_{\gamma,j,\eta_R}\|_{L^{a,1}}.
		\end{equation}
		For the directional kernel, changing variables $h=Rt$
		gives
		$$
		K_{\gamma,j,\eta_R}(x)
		=R^{-2\gamma}K_{\gamma,j,\eta}(x/R).
		$$
		By the definition of Lorentz quasi-norm and Lemma \ref{lem:decay_cutoff_kernel}, we have
		\begin{equation}\label{eq:kernel-scaling}
			\|K_{\gamma,j,\eta_R}\|_{L^{a,1}}
			=R^{n/a-2\gamma}\|K_{\gamma,j,\eta}\|_{L^{a,1}}
			\le CR^{n-2n/r-2\gamma}.
		\end{equation}
	Since $r<2n/(n-2\gamma)$, \eqref{eq:lorentz-pairing} and \eqref{eq:kernel-scaling} give \eqref{eq:2.9}.
		
		For the full kernel, changing variables $y=Rz$ gives
		$K_{\gamma,\eta_R}(x)=R^{-2\gamma}K_{\gamma,\eta}(x/R)$.
		Its $L^{a,1}$ integrability was proved in
		Lemma~\ref{lem:decay_cutoff_kernel}, so the same argument proves the isotropic estimate.

		It remains to consider the endpoint $r=2n/(n-2\gamma)$.
		By Lemma~\ref{lem:decay_cutoff_kernel},
		\[
		K_{\gamma,j,\eta},\ K_{\gamma,\eta}\in L^{a,1}(\mathbb R^n)
		\subset L^a(\mathbb R^n).
		\]
		The scaling identities above therefore imply 
		\[	\sup_{R\ge1}
		\left(
		\|K_{\gamma,j,\eta_R}\|_{L^a}
		+
		\|K_{\gamma,\eta_R}\|_{L^a}
		\right)
		<\infty.
		\]
		Let $K_R$ denote either $K_{\gamma,j,\eta_R}$ or
		$K_{\gamma,\eta_R}$. For $L>0$, write
		\[
		\int_{\mathbb R^n}|v|^2K_R\,\dd x
		=
		\int_{B_L}|v|^2K_R\,\dd x
		+
		\int_{\mathbb R^n\setminus B_L}|v|^2K_R\,\dd x.
		\]
		Since the unscaled kernel is bounded, its scaling gives
		$
		\|K_R\|_{L^\infty}\le C R^{-2\gamma}.
		$
		Therefore, for every fixed $L>0$,
		\[
		\int_{B_L}|v|^2K_R\,\dd x
		\le
		C R^{-2\gamma}\int_{B_L}|v|^2\,\dd x
		\longrightarrow0
		\qquad\text{as }R\to\infty.
		\]
		Moreover, since $r/2=a'$ and
		$\sup_{R\ge1}\|K_R\|_{L^a}<\infty$, H\"older's inequality gives
		\[
		\int_{\mathbb R^n\setminus B_L}|v|^2K_R\,\dd x
		\le
		\|v\|_{L^r(\mathbb R^n\setminus B_L)}^2
		\|K_R\|_{L^a}
		\le
		C\|v\|_{L^r(\mathbb R^n\setminus B_L)}^2.
		\]
		Consequently, for every fixed $L>0$,
		\[
		\limsup_{R\to\infty}
		\int_{\mathbb R^n}|v|^2K_R\,\dd x
		\le
		C\|v\|_{L^r(\mathbb R^n\setminus B_L)}^2.
		\]
		Because $v\in L^r(\mathbb R^n)$, the right-hand side tends to zero
		as $L\to\infty$. Hence
		\[
		\limsup_{R\to\infty}
		\int_{\mathbb R^n}|v|^2K_R\,\dd x=0.
		\]
		The integrals are nonnegative, so
		\[
		\int_{\mathbb R^n}|v|^2K_R\,\dd x\longrightarrow0
		\qquad\text{as }R\to\infty.
		\]
	\end{proof}

	\subsection{Classical localized energy identities}
	
	At dissipation exponent one, the fractional localization formulas
	are replaced by ordinary integration by parts.
	
	\begin{lemma}
		\label{lem:classical-localized}
		Let $f\in C^\infty(\R^n)$ and let
		$\eta\in C_c^\infty(\R^n)$ be real-valued. Then, for every
		$j\in\{1,\ldots,n\}$,
		\begin{equation}\label{eq:classical-localized-directional}
			\pair{-\partial_j^2f}{\eta^2f}
			=
			\|\partial_j(\eta f)\|_2^2
			-
			\int_{\R^n}|\partial_j\eta|^2|f|^2\,\dd x.
		\end{equation}
		Likewise,
		\begin{equation}\label{eq:classical-localized-full}
			\pair{-\Delta f}{\eta^2f}
			=
			\|\nabla(\eta f)\|_2^2
			-
			\int_{\R^n}|\nabla\eta|^2|f|^2\,\dd x.
		\end{equation}
		If, in addition, $f\in L^{r,\infty}(\R^n)$ for $2<r<\frac{2n}{n-2}$,
		then the cutoffs $\eta_R(x)=\eta(x/R)$ fixed above satisfy
		\begin{equation}\label{eq:classical-cutoff-error}
			\int_{\R^n}|\partial_j\eta_R|^2|f|^2\,\dd x
			\le
			C R^{n-2-2n/r}
			\|f\|_{L^{r,\infty}(\R^n)}^2
			\longrightarrow0
			\qquad\text{as }R\to\infty.
		\end{equation}
		If $f\in L^r(\R^n)$ for $r=\frac{2n}{n-2}$, we also have
		\begin{equation}\label{eq:classical-cutoff-end}
			\int_{\R^n}|\partial_j\eta_R|^2|f|^2\,\dd x
			\le
			C R^{n-2-2n/r}
			\|f\|_{L^{r}(\R^n)}^2
			\longrightarrow0
			\qquad\text{as }R\to\infty.
		\end{equation}
	\end{lemma}
	
	\begin{proof}
		The identities	\eqref{eq:classical-localized-directional} and
		\eqref{eq:classical-localized-full} follow directly from integrating by parts.
		For the cutoff estimate, recall that
		\[
		\supp\nabla\eta_R\subset A_R=B_{2R}\setminus B_R,
		\qquad
		\|\nabla\eta_R\|_\infty\le \frac{C}{R}.
		\]
		Since $r>2$, H\"older's inequality gives
		\[
		\int_{A_R}|f|^2\,\dd x
		\le
		C_r |A_R|^{1-2/r}
		\|f\|_{L^{r,\infty}(A_R)}^2.
		\]
		Consequently,
		\[
		\int_{\R^n}|\nabla\eta_R|^2|f|^2\,\dd x
		\le
		\frac{C}{R^2}\int_{A_R}|f|^2\,\dd x\le
		C R^{n-2-2n/r}
		\|f\|_{L^{r,\infty}(\R^n)}^2.
		\]
		The exponent $n-2-2n/r$ is strictly negative because
		$r<2n/(n-2)$. This proves
		\eqref{eq:classical-cutoff-error}.
	\end{proof}

	\section{Proof of Theorem \ref{thm:main}}

	\subsection{Representation of Pressure}
	
	Define
	\begin{equation}\label{eq:pressure0}
		P_0=\sum_{k,\ell=1}^n
		\mathcal R_k\mathcal R_\ell(u_ku_\ell-b_kb_\ell)
		=P_u+P_b.
	\end{equation}
	Lemma~\ref{lem:riesz} and Lorentz H\"older give
	\begin{equation}\label{eq:pressure-bounds}
		\|P_u\|_{L^{p/2,\infty}}\le C\|u\|_{L^{p,\infty}}^2,
		\qquad
		\|P_b\|_{L^{q/2,\infty}}\le C\|b\|_{L^{q,\infty}}^2.
	\end{equation}
	
	\begin{lemma}\label{lem:pressure}
		Under the assumptions of Theorem~\ref{thm:main},
		$P-P_0$ is constant.
	\end{lemma}
	
	\begin{proof}
		We prove the directional case. Taking divergence of the
		momentum equation gives
		$$
		-\Delta P=\sum_{k,\ell}\partial_k\partial_\ell
		(u_ku_\ell-b_kb_\ell).
		$$
		The multiplier identity for the Riesz transforms gives
		the same equation for $P_0$. Thus $h=P-P_0$ is harmonic,
		and
		\begin{equation}\label{eq:grad-h}
			\nabla h=-\Lambda_i^{2\alpha}u
			-\diver(u\otimes u-b\otimes b)-\nabla P_0.
		\end{equation}
		The right-hand side is tempered. Hence each component
		of $\nabla h$ is a tempered harmonic distribution and
		therefore a polynomial.
		
		Fix $\phi\in C_c^\infty(\R^n)$ and set
		$\phi_R(x)=R^{-n}\phi(x/R)$. Scaling gives
		$$
		\|\Lambda_i^{2\alpha}\phi_R\|_{L^{p',1}}
		=R^{-2\alpha-n/p}
		\|\Lambda_i^{2\alpha}\phi\|_{L^{p',1}},
		$$
		for $0<\alpha<1$; when $\alpha=1$, $\Lambda_i^2=-\partial_i^2$
		and
		$\|\partial_i^2\phi_R\|_{L^{p',1}}=R^{-2-n/p}\|\partial_i^2\phi\|_{L^{p',1}}$.
		In both cases the exponent is strictly negative, so
		$$
		\vert{}\pair{\Lambda_i^{2\alpha}u}{\phi_R}\vert{}
		\le C_\phi\|u\|_{L^{p,\infty}}R^{-2\alpha-n/p}
		\longrightarrow0.
		$$
		For the tensor and pressure terms, Lorentz H\"older gives
		$$
		\begin{aligned}
			\vert{}\pair{\diver(u\otimes u)}{\phi_R}\vert{}
			+\vert{}\pair{\nabla P_u}{\phi_R}\vert{}
			&\le C_\phi R^{-1-2n/p}\|u\|_{L^{p,\infty}}^2,\\
			\vert{}\pair{\diver(b\otimes b)}{\phi_R}\vert{}
			+\vert{}\pair{\nabla P_b}{\phi_R}\vert{}
			&\le C_\phi R^{-1-2n/q}\|b\|_{L^{q,\infty}}^2.
		\end{aligned}
		$$
		Here we used
		$\|\nabla\phi_R\|_{L^{(p/2)',1}}
		=R^{-1-2n/p}\|\nabla\phi\|_{L^{(p/2)',1}}$,
		and similarly for $q$.
		Consequently, 
		\begin{equation}\label{eq:grad-h-vanish}
			\pair{\nabla h}{\phi_R}\to0 \quad\text{for every } \phi \in C_c^\infty(\R^n).
		\end{equation}	
		It remains to show that $\nabla h$ vanishes. Suppose, for
		contradiction, that one component of $\nabla h$ is a nonzero
		polynomial. Denote this component by $Q$ and let $m\ge0$ be its
		degree. Write
		\[
		Q=Q_m+Q_{m-1}+\cdots+Q_0,
		\]
		where each $Q_k$ is homogeneous of degree $k$ and $Q_m\not\equiv0$.
		
		Since $Q_m$ is not identically zero, one may choose $\varphi\in C_c^\infty(\R^n)$ supported in a region where $Q_m$ does not vanish and with the same sign as $Q_m$ there, such that
		\[
		\int_{\R^n}Q_m(y)\varphi(y)\,\dd y\neq0.
		\]
		Let $\phi\in C_c^\infty(\R^n)$ have $\varphi$ in the
		corresponding component and all other components equal to zero.
		Then
		\[
		\pair{\nabla h}{\phi_R}
		=
		\int_{\R^n}Q(x)R^{-n}\varphi(x/R)\,\dd x=
		\int_{\R^n}Q(Ry)\varphi(y)\,\dd y=
		\sum_{k=0}^m
		R^k\int_{\R^n}Q_k(y)\varphi(y)\,\dd y.
		\]
		If $m\ge1$, the coefficient of the highest power $R^m$ is
		nonzero, so this expression cannot converge to zero as
		$R\to\infty$. If $m=0$, then $Q$ is a nonzero constant and
		\[
		\pair{\nabla h}{\phi_R}
		=
		Q_0\int_{\R^n}\varphi(y)\,\dd y,
		\]
		which is again nonzero for a suitable choice of $\varphi$.
		
		This contradicts \eqref{eq:grad-h-vanish}.
		Therefore every component of $\nabla h$ is the zero polynomial,
		and hence $\nabla h=0$.
		Thus $h$ is constant, so $P=P_0+C$
		for some constant $C$.
		
		For the full operator, replace $\Lambda_i^{2\alpha}$
		by $\Lambda^{2\alpha}$. Its test-function images have
		the same dilation exponent and belong to $L^{p',1}$.
		All other steps are unchanged.
	\end{proof}
	
	\subsection{Proof of Theorem \ref{thm:main}:directional dissipation}
	\subsubsection{ fractional case
		$0<\alpha,\beta<1$}
	\label{sec:proof-directional-fractional}

	Test \eqref{eq:system1}$_1$ by $\eta_R^2u$ and \eqref{eq:system2}$_2$ by $\eta_R^2b$, and add them together. Using
	$\diver u=\diver b=0$ and
	Lemma~\ref{lem:pressure}, we obtain
	\begin{align}\label{eq:energy-identity}
			&\pair{\Lambda_i^{2\alpha}u}{\eta_R^2u}
			+\pair{\Lambda_j^{2\beta}b}{\eta_R^2b}\nonumber\\
			=&\frac12\int_{A_R}|u|^2u\cdot\nabla(\eta_R^2)\,\dd x
			+\frac12\int_{A_R}|b|^2u\cdot\nabla(\eta_R^2)\,\dd x+\int_{A_R}P_0u\cdot\nabla(\eta_R^2)\,\dd x
			-\int_{A_R}(u\cdot b)b\cdot\nabla(\eta_R^2)\,\dd x\nonumber\\
			=:&J_1+J_2+J_3+J_4.
	\end{align}
	Since $3<p<3n/(n-1)$, Lemma~\ref{lem:holder} gives
	\begin{equation}\label{eq:u-flux}
		|J_1|\le CR^{n-1-3n/p}
		\|u\|_{L^{p,\infty}(A_R)}^3\longrightarrow0.
	\end{equation}
	Since $1/p+2/q>(n-1)/n$,
	\begin{equation}\label{eq:mixed-flux}
		|J_2|+|J_4|
		\le CR^{n-1-n(1/p+2/q)}
		\|u\|_{L^{p,\infty}(A_R)}
		\|b\|_{L^{q,\infty}(A_R)}^2
		\longrightarrow0.
	\end{equation}
	For the pressure term, use \eqref{eq:pressure-bounds}
	and the global pressure norms:
	\begin{align}
		\frac1R\int_{A_R}|P_u||u|\,\dd x
		&\le CR^{n-1-3n/p}
		\|u\|_{L^{p,\infty}(\R^n)}^2
		\|u\|_{L^{p,\infty}(A_R)}
		\longrightarrow0,
		\label{eq:Pu-flux}\\
		\frac1R\int_{A_R}|P_b||u|\,\dd x
		&\le CR^{n-1-n(1/p+2/q)}
		\|b\|_{L^{q,\infty}(\R^n)}^2
		\|u\|_{L^{p,\infty}(A_R)}
		\longrightarrow0.
		\label{eq:PB-flux}
	\end{align}
	Combibing the estimates \eqref{eq:u-flux}--\eqref{eq:PB-flux} and letting $R\to\infty$ in \eqref{eq:energy-identity}, we obtain
	\begin{equation}\label{eq:energy-pairing-zero}
		\pair{\Lambda_i^{2\alpha}u}{\eta_R^2u}
		+\pair{\Lambda_j^{2\beta}b}{\eta_R^2b}
		\longrightarrow0.
	\end{equation}
	Apply Lemmas~\ref{lem:localized} and Lemma \ref{lem:cutoff-error}
	to the two fields. Their cutoff errors tend to zero, and
	their localized energies are nonnegative. Therefore
	\begin{equation}\label{eq:localized-dissipation-zero}
		\|\Lambda_i^\alpha(\eta_Ru)\|_2^2
		+\|\Lambda_j^\beta(\eta_Rb)\|_2^2
		\longrightarrow0.
	\end{equation}
	Let $\Phi\in C_c^\infty(\R_{\xi_i}\times\R^{n-1}_{x'})$
	be supported where $\vert{}\xi_i\vert{}\ge\delta>0$.
	Choose $\psi\in\Sch(\R^n)$ whose partial Fourier
	transform satisfies
	$$
	\mathcal F_{x_i}\psi(\xi_i,x')
	=\vert{}\xi_i\vert{}^{-\alpha}\Phi(\xi_i,x').
	$$
	Then $\mathcal F_{x_i}(\Lambda_i^\alpha\psi)=\Phi$.
	Since $\eta_Ru\to u$ in $\Sch'$,
	self-adjointness and \eqref{eq:localized-dissipation-zero}
	give
	$$
	\vert{}\pair{\eta_R u}{\Lambda_i^\alpha\psi}\vert{}
	=\vert{}\pair{\Lambda_i^\alpha (\eta_R u)}{\psi}\vert{}
	\le\|\Lambda_i^\alpha (\eta_R u)\|_2\|\psi\|_2
	\longrightarrow0.
	$$
	Hence $\langle u,\Lambda_i^\alpha\psi\rangle=0$. By the definition of $\psi$ and Plancherel identity in $x_1$-variable, we have
	\[
	0=\langle u,\Lambda_i^\alpha\psi\rangle
	=
	\int_{\R^{n-1}}\int_{\R}
	\mathcal F_{x_i}u(\xi_i,x')
	\overline{\Phi(\xi_i,x')}\,d\xi_i\,dx'.
	\]
	This holds for every $\Phi\in C_c^\infty(\R_{\xi_i}\times\R^{n-1}_{x'})$ supported where $|\xi_i|\ge \delta>0$.
	Thus $\mathcal F_{x_i}u$ annihilates every test function supported in the region $|\xi_i|\ge \delta$. Hence
	$\mathcal F_{x_i}u=0$ on $\{\,|\xi_i|\ge \delta\,\}.$
	Since $\delta>0$ is arbitrary, $\mathcal F_{x_i}u$ cannot have support at any point with $\xi_i\neq 0$. Therefore
	\begin{equation}\label{eq:u-support}
		\supp \mathcal F_{x_i}u
		\subset \{\xi_i=0\}.
	\end{equation}
	Applying the same argument to $b$, with $j,\beta$
	in place of $i,\alpha$, gives
	\begin{equation}\label{eq:b-support}
		\supp\mathcal F_{x_j}b
		\subset\{\xi_j=0\}.
	\end{equation}
	Fix a component $u_k$ and
	$\theta\in C_c^\infty(\R^{n-1})$. Define
	$$
	g_\theta(t)=\int_{\R^{n-1}}u_k(t,x')\theta(x')\,\dd x',
	\qquad t=x_i.
	$$
	Let $\lambda>0$. Define $E_\lambda:=\{t\in\R: |g_\theta(t)|>\lambda\}$ and
	\[
	E_{\lambda,m}:=E_\lambda\cap(-m,m)
	=\{t\in(-m,m): |g_\theta(t)|>\lambda\}.
	\]
	Applying Lemma \ref{lem:holder}, we have
	\[
	\int_{E_{\lambda,m}} |g_\theta(t)|\dd t
	\le C\|u\|_{L^{p,\infty}(\R^n)}
	\|\theta\|_{L^{p',1}(\R^{n-1})}
	|E_{\lambda,m}|^{1/p'}.
	\]
	Since $|g_\theta(t)|>\lambda$ on $E_{\lambda,m}$, we also have
	\[
	\int_{E_{\lambda,m}} |g_\theta(t)|\dd t
	\ge \lambda |E_{\lambda,m}|.
	\]
	Hence
	\[
	\lambda |E_{\lambda,m}|
	\le C\|u\|_{L^{p,\infty}(\R^n)}
	\|\theta\|_{L^{p',1}(\R^{n-1})}
	|E_{\lambda,m}|^{1/p'}.
	\]
	Therefore,
	\[
	\lambda |E_{\lambda,m}|^{1/p}
	\le C\|u\|_{L^{p,\infty}(\R^n)}
	\|\theta\|_{L^{p',1}(\R^{n-1})}.
	\]
	Equivalently,
	\[
	|E_{\lambda,m}|
	\le
	\left(
	\frac{C\|u\|_{L^{p,\infty}(\R^n)}
		\|\theta\|_{L^{p',1}(\R^{n-1})}}{\lambda}
	\right)^p.
	\]
	The right-hand side is independent of $m$. Since
	\[
	E_{\lambda,m}\uparrow E_\lambda
	\quad\text{as } m\to\infty,
	\]
	monotone convergence gives
	\[
	|E_\lambda|
	\le
	\left(
	\frac{C\|u\|_{L^{p,\infty}(\R^n)}
		\|\theta\|_{L^{p',1}(\R^{n-1})}}{\lambda}
	\right)^p
	<\infty.
	\]
	Equivalently,
	\[
	\|g_\theta\|_{L^{p,\infty}(\R)}
	\le C\|u\|_{L^{p,\infty}(\R^n)}
	\|\theta\|_{L^{p',1}(\R^{n-1})}.
	\]
	Since $\widehat{g_\theta}(\xi_i)
	=
	\int_{\mathbb{R}^{n-1}}
	\mathcal{F}_{x_i} u_k(\xi_i, x') \theta(x') \, dx'$, by \eqref{eq:u-support},
	$\supp\widehat g_\theta\subset\{0\}$. Since $g_\theta\in L^{p,\infty}(\R^n)$ is tempered, $g_\theta$
	is a polynomial. A nonzero polynomial cannot belong to
	$L^{p,\infty}(\R)$ for finite $p$, so $g_\theta=0$.
	As $\theta$ and $k$ are arbitrary, $u=0$.
	The same argument, using \eqref{eq:b-support} and
	$q<\infty$, gives $b=0$.
	The momentum equation then gives $\nabla P=0$.
	
	\subsubsection{classical endpoints}
	\label{sec:proof-directional-endpoint}
	
	We now prove the three remaining cases:
	$$
	0<\alpha<1,\ \beta=1;
	\qquad
	\alpha=1,\ 0<\beta<1;
	\qquad
	\alpha=\beta=1.
	$$	
	For the classical directional operator, Lemma~\ref{lem:classical-localized}
	gives
	\begin{equation}\label{eq:classical-directional-energy}
		\pair{-\partial_j^2f}{\eta_R^2f}
		=
		\|\partial_j(\eta_Rf)\|_2^2
		-\int_{\R^n}|\partial_j\eta_R|^2|f|^2\,\dd x.
	\end{equation}
	If $f\in L^{r,\infty}$ and $2<r<2n/(n-2)$,
	then \eqref{eq:classical-cutoff-error} implies
	\begin{equation}\label{eq:classical-error-limit}
		\int_{\R^n}|\partial_j\eta_R|^2|f|^2\,\dd x
		\longrightarrow0.
	\end{equation}
	
	\noindent{\bf Case 1:  $0<\alpha<1$ and $\beta=1$}
	
	Assume now $0<\alpha<1$ and $\beta=1$.
	The exponent assumptions in Theorem~\ref{thm:main} imply
	$$
	p<\frac{2n}{n-2\alpha},
	\qquad
	q<\frac{2n}{n-2}.
	$$
	The estimates \eqref{eq:u-flux}--\eqref{eq:PB-flux} give
	$$
	\pair{\Lambda_i^{2\alpha}u}{\eta_R^2u}
	+
	\pair{-\partial_j^2b}{\eta_R^2b}
	\longrightarrow0.
	$$
	For the first term, Lemmas~\ref{lem:localized} and
	\ref{lem:cutoff-error} imply
	$$
	\|\Lambda_i^\alpha(\eta_Ru)\|_2^2
	\le
	\pair{\Lambda_i^{2\alpha}u}{\eta_R^2u}+o(1).
	$$
	For the second term, \eqref{eq:classical-directional-energy} and
	\eqref{eq:classical-error-limit} imply
	$$
	\|\partial_j(\eta_Rb)\|_2^2
	\le
	\pair{-\partial_j^2b}{\eta_R^2b}+o(1).
	$$
	Therefore,
	\begin{equation}\label{eq:mixed-alpha-fractional-beta-classical}
		\|\Lambda_i^\alpha(\eta_Ru)\|_2^2
		+\|\partial_j(\eta_Rb)\|_2^2
		\longrightarrow0.
	\end{equation}
	
	The first term in \eqref{eq:mixed-alpha-fractional-beta-classical}
	implies
	$$
	\supp\mathcal F_{x_i}u\subset\{\xi_i=0\}.
	$$
	For the second term, take any $\Psi\in C_c^\infty$ supported where $\vert{}\xi_j\vert{}\ge\delta>0$.
	Choose $\psi\in\mathcal S$ such that $\Psi = i\xi_j \widehat{\psi}$. Then
	$$
	\left\vert{}\pair{\eta_Rb}{\Psi}\right\vert{}
	=
	\left\vert{}\pair{\partial_j(\eta_Rb)}{\psi}\right\vert{}
	\le
	\|\partial_j(\eta_Rb)\|_2\|\psi\|_2
	\longrightarrow0,
	$$
	which gives
	$$
	\supp\mathcal F_{x_j}b\subset\{\xi_j=0\}.
	$$
	The weak-Lebesgue argument used above now gives $u=b=0$.
	Finally, the momentum equation yields $\nabla P=0$.
	
	\noindent{\bf Case 2: $\alpha=1$ and $0<\beta<1$}
	
	Assume $\alpha=1$ and $0<\beta<1$.
	The exponent assumptions imply
	\[
	p<\frac{2n}{n-2},
	\qquad
	q<\frac{2n}{n-2\beta}.
	\]
	Use \eqref{eq:classical-directional-energy} and
	\eqref{eq:classical-error-limit} for $u$, and
	Lemmas~\ref{lem:localized} and \ref{lem:cutoff-error} for $b$.
	The  estimates then give
	\[
	\|\partial_i(\eta_Ru)\|_2^2
	+
	\|\Lambda_j^\beta(\eta_Rb)\|_2^2
	\longrightarrow0.
	\]
	Testing against Fourier-localized functions supported away from
	$\{\xi_i=0\}$ and $\{\xi_j=0\}$ gives
	\[
	\supp\mathcal F_{x_i}u\subset\{\xi_i=0\},
	\qquad
	\supp\mathcal F_{x_j}b\subset\{\xi_j=0\}.
	\]
	The weak-Lebesgue argument used above implies $u=b=0$.
	Consequently, $\nabla P=0$.
	
	\noindent{\bf Case 3: $\alpha=\beta=1$}
	
	Assume $\alpha=\beta=1$. The exponent assumptions give
	\[
	p<\frac{2n}{n-2},
	\qquad
	q<\frac{2n}{n-2}.
	\]
	Using the classical localized identity for both fields and the
	estimates involving $\nabla \eta_R$ yields
	\[
	\|\partial_i(\eta_Ru)\|_2^2
	+
	\|\partial_j(\eta_Rb)\|_2^2
	\longrightarrow0.
	\]
	Testing against Fourier-localized functions away from the
	respective coordinate hyperplanes gives
	\[
	\supp\mathcal F_{x_i}u\subset\{\xi_i=0\},
	\qquad
	\supp\mathcal F_{x_j}b\subset\{\xi_j=0\}.
	\]
	Finite weak-Lebesgue integrability excludes nonzero polynomial
	dependence in the remaining coordinate variables, so $u=b=0$.
	The momentum equation gives $\nabla P=0$.
	
	This completes the directional part of
	Theorem~\ref{thm:main} for all $0<\alpha,\beta\le1$.
	
	\subsection{Proof of Theorem~\ref{thm:main}:
		full dissipation}
	\label{sec:proof-full}
	
	The pressure normalization and the flux estimates
	\eqref{eq:u-flux}--\eqref{eq:PB-flux} apply to the full-dissipation
	system as well. In particular,
	\[
	\pair{\Lambda^{2\alpha}u}{\eta_R^2u}
	+
	\pair{\Lambda^{2\beta}b}{\eta_R^2b}
	\longrightarrow0.
	\]
	If $0<\alpha,\beta<1$, apply the full-operator parts of
	Lemmas~\ref{lem:localized} and \ref{lem:cutoff-error}. This gives
	\[
	\|\Lambda^\alpha(\eta_Ru)\|_2^2
	+
	\|\Lambda^\beta(\eta_Rb)\|_2^2
	\longrightarrow0.
	\]
	If $\alpha=1$, use the full classical identity in
	Lemma~\ref{lem:classical-localized} for $u$; if $\beta=1$, use
	that identity for $b$. The corresponding cutoff errors tend to
	zero by the full version of \eqref{eq:classical-cutoff-error}.
	Thus, in all four cases $0<\alpha,\beta\le1$,
	\[
	\|\Lambda^\alpha(\eta_Ru)\|_2^2
	+
	\|\Lambda^\beta(\eta_Rb)\|_2^2
	\longrightarrow0,
	\]
	where $\|\Lambda f\|_2=\|\nabla f\|_2$ at exponent $1$.
	
	For $\Phi\in C_c^\infty(\R^n\setminus\{0\};\R^n)$, choose
	$\psi\in\Sch(\R^n;\R^n)$ such that $\widehat\psi(\xi)=|\xi|^{-\alpha}\Phi(\xi)$.
	Then
	\[
	\left|\pair{\eta_Ru}{\Lambda^\alpha\psi}\right|
	=
	\left|\pair{\Lambda^\alpha(\eta_Ru)}{\psi}\right|
	\le
	\|\Lambda^\alpha(\eta_Ru)\|_2\|\psi\|_2
	\longrightarrow0.
	\]
	For $\alpha=1$, the same identity is interpreted using the
	ordinary gradient and the Fourier multiplier $|\xi|$.
	Since $\eta_Ru\to u$ in $\Sch'$, it follows that
	\[
	\supp\widehat u\subset\{0\}.
	\]
	The identical argument gives $\supp\widehat b\subset\{0\}$.
	Thus $u$ and $b$ are polynomial distributions. Since they belong
	to finite weak-Lebesgue spaces, both vanish. The momentum equation
	then gives $\nabla P=0$. This proves the full-dissipation case.
	
	\section{Proof of Theorem \ref{thm:stability}}
	\label{sec:stability-proof}
	\subsection{Perturbation stress and pressure}
	
	Define
	\begin{equation}\label{eq:stab-stress}
		F=u_0\otimes v+v\otimes u_0+v\otimes v
		-b_0\otimes w-w\otimes b_0-w\otimes w,
	\end{equation}
	with convention
	$(\diver F)_m=\sum_\ell\partial_\ell F_{m\ell}$.
	Every term on the right-hand side of \eqref{eq:stab-stress} contains a factor from $\{v,w\}$, and
	$v,w\in C_c^\infty(\R^n)$. Since $U$ and $B$ are smooth,
	it follows that $F\in C_c^\infty(\R^n)$
	regardless of the growth of $U$ and $B$.
	The momentum equation becomes
	\begin{equation}\label{eq:stab-conservative}
		\Lambda_1^{2\alpha}v+\diver F+\nabla P=0.
	\end{equation}
	Set
	$$
	Q=\sum_{m,\ell}\mathcal R_m\mathcal R_\ell F_{m\ell}.
	$$
	
	\begin{lemma}\label{lem:stab-pressure}
		For a compactly supported perturbation as in
		Theorem~\ref{thm:stability}, $P=Q+C$ for some constant $C$.
	\end{lemma}
	
	\begin{proof}
		We have $Q\in L^2$ and $H=P-Q$ harmonic.
		Equation~\eqref{eq:stab-conservative} shows that
		$\nabla H$ is tempered, hence polynomial.
		For $\phi_R=R^{-n}\phi(\cdot/R)$,
		$$
		\vert{}\pair{\Lambda_1^{2\alpha}v}{\phi_R}\vert{}
		\le R^{-2\alpha-n/2}\|v\|_2
		\|\Lambda_1^{2\alpha}\phi\|_2,
		$$
		and
		$$
		\vert{}\pair{\diver F+\nabla Q}{\phi_R}\vert{}
		\le C_\phi R^{-1-n/2}(\|F\|_2+\|Q\|_2).
		$$
		Thus $\pair{\nabla H}{\phi_R}\to0$.
		The polynomial argument in Lemma~\ref{lem:pressure}
		gives $\nabla H=0$.
	\end{proof}
	
	\subsection{Perturbation energy identity}
	
	Define
	\begin{align}
		D(v,w)&=\|\Lambda_1^\alpha v\|_2^2
		+\|\Lambda_1^\beta w\|_2^2,
		\label{eq:stab-dissipation}\\
		I(v,w)&=\int_{\R^n}
		\bigl[v_1v\cdot\nabla U-w_1w\cdot\nabla U
		+w_1v\cdot\nabla B-v_1w\cdot\nabla B\bigr]\,\dd x.
		\label{eq:stab-interaction}
	\end{align}
	The transverse gradients are embedded in $\R^n$
	with first component zero.
	
	\begin{proposition}\label{prop:stab-energy}
		Every smooth compactly supported solution of
		\eqref{eq:perturbation} satisfies
		\begin{equation}\label{eq:stab-energy}
			D(v,w)+I(v,w)=0.
		\end{equation}
	\end{proposition}
	
	\begin{proof}
		Test the two equations by $v,w$ and add.
		Plancherel gives the two terms in $D$.
		The pressure vanishes by $\diver v=0$.
		For divergence-free $a$ and compactly supported $f$,
		$$
		\int(a\cdot\nabla)f\cdot f\,\dd x=0.
		$$
		This removes the transport terms involving $u_0$
		and the corresponding nonlinear velocity terms.
		The paired magnetic transport terms cancel because
		\[
		\int_{\R^n}
		\bigl[(a\cdot\nabla)w\cdot v
		+(a\cdot\nabla)v\cdot w\bigr]\,\dd x
		=
		\int_{\R^n}a\cdot\nabla(v\cdot w)\,\dd x=0
		\]
		whenever $a$ is divergence-free and the fields are compactly
		supported. This applies with $a=b_0$ and $a=w$.
		
		The remaining background-gradient terms are
		\[
		\int_{\R^n}
		\bigl[
		(v\cdot\nabla)u_0\cdot v
		-(w\cdot\nabla)u_0\cdot w
		+(v\cdot\nabla)b_0\cdot w
		-(w\cdot\nabla)b_0\cdot v
		\bigr]\,\dd x.
		\]
		Because
		\[
		u_0=U(x')e_1,\qquad b_0=B(x')e_1,
		\]
		we have
		\[
		(v\cdot\nabla)u_0
		=(v\cdot\nabla' U)e_1,
		\qquad
		(w\cdot\nabla)b_0
		=(w\cdot\nabla' B)e_1.
		\]
		Using $v\cdot\nabla' U=v_1\,v\cdot\nabla U$ in the
		energy pairing and the analogous identity for $B$, the preceding
		expression is exactly $I(v,w)$. Hence
		\[
		D(v,w)+I(v,w)=0.
		\]
	\end{proof}
	
	\subsection{Proof of Theorem~\ref{thm:stability}}
	
	We first prove that $w=0$. The second equation in
	\eqref{eq:perturbation} can be written as
	\[
	\Lambda_1^{2\beta}w=G,
	\]
	where
	\[ 	G=-(u_0\cdot\nabla)w-(v\cdot\nabla)b_0
	-(v\cdot\nabla)w +(b_0\cdot\nabla)v+(w\cdot\nabla)u_0
	+(w\cdot\nabla)v.
	\]
	Every term in $G$ is smooth and compactly supported in
	$\R^n$. 
	Fix $x'\in\R^{n-1}$. Then
	\[
	w(\,\cdot\,,x')\in C_c^\infty(\R),
	\qquad
	\Lambda_1^{2\beta}w(\,\cdot\,,x')
	=G(\,\cdot\,,x')\in C_c^\infty(\R).
	\]
	Applying Lemma~\ref{lem:compact-support} componentwisely, we have 
	\[
	w(\,\cdot\,,x')=0.
	\]
	Since $x'$ is arbitrary, $w=0$ in $\R^n$.
	
	With $w=0$, define
	\[
	F=u_0\otimes v+v\otimes u_0+v\otimes v.
	\]
	The first equation becomes
	\begin{equation}\label{eq:compact-reduced-velocity}
		\Lambda_1^{2\alpha}v+\diver F+\nabla P=0.
	\end{equation}
	Here $F\in C_c^\infty(\R^n)$. Let
	\[
	\omega_{k\ell}
	=\partial_kv_\ell-\partial_\ell v_k.
	\]
	Applying $\partial_k$ to the $\ell$-th component of
	\eqref{eq:compact-reduced-velocity}, applying $\partial_\ell$ to
	the $k$-th component, and subtracting, gives
	\[
	\Lambda_1^{2\alpha}\omega_{k\ell}
	=
	-\partial_k(\diver F)_\ell
	+\partial_\ell(\diver F)_k.
	\]
	The right-hand side is in $C_c^\infty(\R^n)$. Therefore, for
	each fixed $x'$, both
	\[
	\omega_{k\ell}(\,\cdot\,,x')
	\quad\text{and}\quad
	\Lambda_1^{2\alpha}\omega_{k\ell}(\,\cdot\,,x')
	\]
	are compactly supported smooth functions of $x_1$. Lemma
	\ref{lem:compact-support} implies
	\[
	\omega_{k\ell}=0
	\qquad\text{for all }k,\ell.
	\]
	Using $\diver v=0$, we have
	\[
	\Delta v_\ell
	=
	\sum_{k=1}^n\partial_k
	\bigl(\partial_kv_\ell-\partial_\ell v_k\bigr)
	+\partial_\ell\diver v
	=0.
	\]
	Thus every component of $v$ is harmonic. Since $v$ is compactly
	supported, we immediately get $v=0$. Substitution into
	\eqref{eq:compact-reduced-velocity} yields $\nabla P=0$.
	This proves Theorem~\ref{thm:stability}.
	
	\section{Proof of Theorem~\ref{thm:intro-noncompact}}
	\label{sec:noncompact}
	We first give a Sobolev-Poincar\'e inequality for directional concentrated functions.
	\begin{lemma}
		\label{lem:directional-concentration}
		Let $s>0$, $0<\kappa<1$, and let $I\subset\R$ be an interval
		of length $\ell>0$. Suppose
		\[
		f\in L^2(\R^n),\qquad \Lambda_1^s f\in L^2(\R^n),
		\]
		and
		\begin{equation}\label{eq:assumption-concentration}
			\int_{\{x_1\in I\}}|f(x)|^2\,\dd x
			\ge \kappa\|f\|_2^2.
		\end{equation}
		Then
		\begin{equation}\label{eq:concentration-poincare}
			\|f\|_2^2
			\le C_{s,\kappa}\ell^{2s}\|\Lambda_1^sf\|_2^2,
			\qquad
			C_{s,\kappa}
			=
			\frac4\kappa
			\left(\frac4{\pi\kappa}\right)^{2s}.
		\end{equation}
		The assertion holds componentwise for vector-valued functions.
	\end{lemma}
	
	\begin{proof}
		Let $P_N$ be the orthogonal projection defined by
		\[
		\widehat{P_N f}(\xi_1,x')
		=
		{\bf1}_{\{|\xi_1|\le N\}}\widehat f(\xi_1,x').
		\]
		Equivalently, 
		\[
		(P_Nf)(x_1,x')
		=
		\frac{1}{\sqrt{2\pi}}
		\int_{|\xi_1|\le N}
		\widehat f(\xi_1,x')e^{ix_1\xi_1}\dd\xi_1.
		\]
		By Cauchy--Schwarz,
		\[
		|(P_Nf)(x_1,x')|^2
		\le
		\frac{1}{2\pi}
		\left(\int_{|\xi_1|\le N}1\dd\xi_1\right)
		\left(\int_{|\xi_1|\le N}
		|\widehat f(\xi_1,x')|^2\dd\xi_1\right)=
		\frac{N}{\pi}
		\int_{|\xi_1|\le N}
		|\widehat f(\xi_1,x')|^2\dd\xi_1.
		\]
		Integrating this inequality in $x'\in\R^{n-1}$, we obtain
		\[
		\|(P_Nf)(x_1,\cdot)\|_{L^2_{x'}}^2
		\le
		\frac{N}{\pi}
		\int_{\R^{n-1}}
		\int_{|\xi_1|\le N}
		|\widehat f(\xi_1,x')|^2\dd\xi_1\dd x'
		\le
		\frac{N}{\pi}\|f\|_2^2.
		\]
		Therefore
		\[
		\sup_{x_1\in\R}
		\|(P_Nf)(x_1,\cdot)\|_{L^2_{x'}}
		\le
		\left(\frac{N}{\pi}\right)^{1/2}\|f\|_2.
		\]
		Since $I$ has length $\ell$,
		\[
		\|{\bf 1}_I P_Nf\|_2^2
		=
		\int_I
		\|(P_Nf)(x_1,\cdot)\|_{L^2_{x'}}^2\dd x_1\le
		\ell
		\sup_{x_1\in\R}
		\|(P_Nf)(x_1,\cdot)\|_{L^2_{x'}}^2\le
		\frac{\ell N}{\pi}\|f\|_2^2.
		\]
		Hence
		\begin{equation}\label{eq:low-frequency}
			\|{\bf 1}_I P_Nf\|_2
			\le
			\left(\frac{\ell N}{\pi}\right)^{1/2}\|f\|_2.
		\end{equation}
		Next we estimate the high-frequency part. Since
		\[
		\widehat{(1-P_N)f}(\xi_1,x')
		=
		{\bf 1}_{\{|\xi_1|>N\}}\widehat f(\xi_1,x'),
		\]
		we have, for $|\xi_1|>N$,
		\[
		|\widehat f(\xi_1,x')|
		\le
		N^{-s}|\xi_1|^s|\widehat f(\xi_1,x')|
		=
		N^{-s}|\widehat{\Lambda_1^s f}(\xi_1,x')|.
		\]
		Thus
		\[
		\|(1-P_N)f\|_2^2
		=
		\int_{|\xi_1|>N}
		|\widehat f(\xi_1,x')|^2\dd\xi_1\dd x'\le
		N^{-2s}\|\Lambda_1^s f\|_2^2.
		\]
		Therefore
		\begin{equation}\label{eq:high-frequency}
			\|(1-P_N)f\|_2
			\le
			N^{-s}\|\Lambda_1^s f\|_2.
		\end{equation}
		The directional concentration assumption \eqref{eq:assumption-concentration} then gives
		\[
		\sqrt{\kappa}\|f\|_2
		\le
		\|{\bf 1}_I P_Nf\|_2+\|(1-P_N)f\|_2.
		\]
		Combining this with \eqref{eq:low-frequency} and \eqref{eq:high-frequency}, we get
		\[
		\sqrt{\kappa}\|f\|_2
		\le
		\left(\frac{\ell N}{\pi}\right)^{1/2}\|f\|_2
		+
		N^{-s}\|\Lambda_1^s f\|_2.
		\]
		\eqref{eq:concentration-poincare} then follows by choosing
		\[
		N=\frac{\pi\kappa}{4\ell}.
		\]
	\end{proof}
	
	\begin{lemma}
		\label{lem:directional-energy-cutoff}
		Let $0<s\le1$, and suppose $f\in L^2(\R^n)$ and $\Lambda_1^sf\in L^2(\R^n)$.
		Then
		\[
		\eta_Rf\longrightarrow f,
		\qquad
		\Lambda_1^s(\eta_Rf)\longrightarrow\Lambda_1^sf
		\quad\text{in }L^2(\R^n)
		\]
		as $R\to\infty$. In particular, with the energy pairing
		\[
		\pair{\Lambda_1^{2s}f}{g}
		:=
		\int_{\R^n}\Lambda_1^sf\cdot\Lambda_1^sg\,\dd x,
		\]
		one has
		\begin{equation}\label{eq:directional-energy-cutoff}
			\pair{\Lambda_1^{2s}f}{\eta_Rf}
			\longrightarrow
			\|\Lambda_1^sf\|_2^2.
		\end{equation}
		\end{lemma}
	
	\begin{proof}
		Since $\eta_R(x)\to\eta(0)=1$ pointwise and
		$\sup_R\|\eta_R\|_\infty<\infty$, dominated convergence gives
		\[
		\|\eta_Rf-f\|_2\longrightarrow0.
		\]
		
		Suppose first that $0<s<1$. For Schwartz functions, the singular
		integral formula for $\Lambda_1^s$ gives
		\[
		[\Lambda_1^s,\eta_R]f(x)
		=
		c_s\int_{\R}
		\frac{\bigl(\eta_R(x)-\eta_R(x+he_1)\bigr)f(x+he_1)}
		{|h|^{1+s}}\,\dd h,
		\]
		where $c_s$ is the normalization constant. The cutoff satisfies
		\[
		|\eta_R(x)-\eta_R(x+he_1)|
		\le C_\eta\min\{|h|/R,1\}.
		\]
		Minkowski's inequality and translation invariance of the $L^2$ norm
		therefore yield
		\begin{equation}\label{eq:commutator}
			\|[\Lambda_1^s,\eta_R]f\|_2
			\le
			C_{s,\eta}\|f\|_2
			\int_{\R}\frac{\min\{|h|/R,1\}}{|h|^{1+s}}\,\dd h\le C_{s,\eta}R^{-s}\|f\|_2.
		\end{equation}
		The estimate extends by density to every $f\in L^2$ for which
		$\Lambda_1^sf\in L^2$. Since
		\[
		\Lambda_1^s(\eta_Rf)-\Lambda_1^sf
		=
		(\eta_R-1)\Lambda_1^sf+[\Lambda_1^s,\eta_R]f,
		\]
		the first term tends to zero in $L^2$ by dominated convergence,
		and the second tends to zero by the commutator estimate \eqref{eq:commutator}. This proves
		\[
		\Lambda_1^s(\eta_Rf)\longrightarrow\Lambda_1^sf
		\quad\text{in }L^2
		\]
		for $0<s<1$.
		
		For $s=1$, $\|\Lambda_1f\|_2=\|\partial_1f\|_2$, and the weak
		product rule gives
		\[
		\partial_1(\eta_Rf)-\partial_1f
		=
		(\eta_R-1)\partial_1f+(\partial_1\eta_R)f.
		\]
		The first term tends to zero in $L^2$ by dominated convergence,
		while
		\[
		\|(\partial_1\eta_R)f\|_2
		\le
		R^{-1}\|\partial_1\eta\|_\infty\|f\|_2
		\longrightarrow0.
		\]
		Thus the asserted convergence also holds for $s=1$.
		
		Finally, by Cauchy--Schwarz,
		\[
		\begin{aligned}
			\left|
			\pair{\Lambda_1^{2s}f}{\eta_Rf}
			-\|\Lambda_1^sf\|_2^2
			\right|
			&=
			\left|
			\int_{\R^n}\Lambda_1^sf\cdot
			\bigl(\Lambda_1^s(\eta_Rf)-\Lambda_1^sf\bigr)\,\dd x
			\right|\\
			&\le
			\|\Lambda_1^sf\|_2
			\|\Lambda_1^s(\eta_Rf)-\Lambda_1^sf\|_2
			\longrightarrow0.
		\end{aligned}
		\]
		This proves \eqref{eq:directional-energy-cutoff}.
	\end{proof}
	
	\subsection{Pressure representation and energy identity}
	
	For the perturbation stress, write
	\[
	F=F_{\mathrm{lin}}+F_{vv}+F_{ww},
	\]
	where
	\[
	\begin{aligned}
		F_{\mathrm{lin}}
		&=u_0\otimes v+v\otimes u_0
		-b_0\otimes w-w\otimes b_0,\\
		F_{vv}&=v\otimes v,
		\qquad
		F_{ww}=-w\otimes w.
	\end{aligned}
	\]
	Define
	\[
	Q_{\mathrm{lin}}
	=
	\sum_{k,\ell=1}^n\mathcal R_k\mathcal R_\ell
	(F_{\mathrm{lin}})_{k\ell},
	\qquad
	Q_{vv}
	=
	\sum_{k,\ell=1}^n\mathcal R_k\mathcal R_\ell
	(v_kv_\ell),
	\]
	and
	\[
	Q_{ww}
	=
	-\sum_{k,\ell=1}^n\mathcal R_k\mathcal R_\ell
	(w_kw_\ell),
	\qquad
	Q=Q_{\mathrm{lin}}+Q_{vv}+Q_{ww}.
	\]
	Since $U,B\in L^\infty$ and $v,w\in L^2$,
	$F_{\mathrm{lin}}\in L^2$. Also
	\[
	F_{vv}\in L^{p/2,\infty},
	\qquad
	F_{ww}\in L^{q/2,\infty}.
	\]
The boundedness of Riesz transform then implies 
	\begin{equation}\label{eq:Q-weak}
	Q_{\mathrm{lin}}\in L^2,\qquad
	Q_{vv}\in L^{p/2,\infty},\qquad Q_{ww}\in L^{q/2,\infty}.
	\end{equation}
	
	\begin{lemma}
		\label{lem:pressure-noncompact}
		Under the assumptions of Theorem~\ref{thm:intro-noncompact},
		\[
		P=Q+C
		\]
		for some constant $C$.
	\end{lemma}
	
	\begin{proof}
		Taking the divergence of the momentum equation gives
		\[
		-\Delta P=\partial_k\partial_\ell F_{k\ell}.
		\]
		The same equation holds for $Q$, so $H=P-Q$ is harmonic.
		The momentum equation also gives
		\[
		\nabla H=-\Lambda_1^{2\alpha}v-\diver F-\nabla Q,
		\]
		which is tempered. Hence every component of $\nabla H$ is a
		tempered harmonic distribution and therefore a polynomial.
		
		For $\phi\in C_c^\infty(\R^n)$, set
		\[
		\phi_R(x)=R^{-n}\phi(x/R).
		\]
		If $0<\alpha<1$, duality and scaling give
		\[
		|\pair{\Lambda_1^{2\alpha}v}{\phi_R}|
		\le
		\|v\|_2\|\Lambda_1^{2\alpha}\phi_R\|_2
		\le
		C_\phi R^{-2\alpha-n/2}\|v\|_2
		\longrightarrow0.
		\]
		For $\alpha=1$, the same estimate follows with
		$\Lambda_1^2=-\partial_1^2$ and exponent $-2-n/2$.
		The linear stress and pressure terms satisfy
		\[
		|\pair{\diver F_{\mathrm{lin}}+\nabla Q_{\mathrm{lin}}}
		{\phi_R}|
		\le
		C_\phi R^{-1-n/2}
		\bigl(\|F_{\mathrm{lin}}\|_2+\|Q_{\mathrm{lin}}\|_2\bigr)
		\longrightarrow0.
		\]
		For the quadratic terms, we have
		\[
		\begin{aligned}
			|\pair{\diver F_{vv}+\nabla Q_{vv}}{\phi_R}|
			&\le
			C_\phi R^{-1-2n/p}
			\bigl(\|F_{vv}\|_{p/2,\infty}+\|Q_{vv}\|_{p/2,\infty}\bigr)
			\longrightarrow0,\\
			|\pair{\diver F_{ww}+\nabla Q_{ww}}{\phi_R}|
			&\le
			C_\phi R^{-1-2n/q}
			\bigl(\|F_{ww}\|_{q/2,\infty}+\|Q_{ww}\|_{q/2,\infty}\bigr)
			\longrightarrow0.
		\end{aligned}
		\]
		Consequently,
		\[
		\pair{\nabla H}{\phi_R}\longrightarrow0
		\qquad\text{for every }\phi\in C_c^\infty(\R^n).
		\]
		A nonzero polynomial cannot have this property: testing against
		a compactly supported function with nonzero pairing with its
		highest-degree homogeneous part gives a nonzero polynomial in $R$.
		Thus $\nabla H=0$, and $P=Q+C$.
	\end{proof}
	
	\begin{proposition}
		\label{prop:noncompact-energy}
		Under the assumptions of Theorem~\ref{thm:intro-noncompact}, we have $\Lambda_1^\alpha v$ and $\Lambda_1^\beta w\in L^2(\R^n)$. Moreover, it holds that
		\[
		\|\Lambda_1^\alpha v\|_2^2
		+\|\Lambda_1^\beta w\|_2^2
		+I(v,w)=0,
		\]
		where
		\[
		I(v,w)=
		\int_{\R^n}
		\left[
		v_1v\cdot\nabla U-w_1w\cdot\nabla U
		+w_1v\cdot\nabla B-v_1w\cdot\nabla B
		\right]\dd x.
		\]
	\end{proposition}
	
	\begin{proof}
		We first show that $\Lambda_1^\alpha v$ and $\Lambda_1^\beta w\in L^2(\R^n)$.
Test the two perturbation equations	by $\eta_Rv$ and $\eta_Rw$, respectively. The divergence-free conditions yield the standard cancellations, thus we have
\begin{equation}\label{eq-energy-nocompact-cut}
	D_R+I_R+E_R^{(1)}+E_R^{(2)}=0,
\end{equation}
where
	\[
\begin{aligned}
	D_R&=\ip{\Lambda_1^{2\alpha}v}{\eta_R^2v}+	\ip{\Lambda_1^{2\beta}w}{\eta_R^2w},\\
	I_R	&=\int_{\R^n}\eta_R^2	\left[
	v_1v'\cdot\nabla' U-w_1w'\cdot\nabla' U
	+w_1v'\cdot\nabla' B-v_1w'\cdot\nabla' B
	\right]\dd x,\\
	E_R^{(1)}&=\int_{\R^n}\Big[	-\tfrac12(|v|^2+|w|^2)u_0\cdot\nabla(\eta_R)^2
	+(v\cdot w)b_0\cdot\nabla(\eta_R)^2\Big]\,\dd x,\\
	E_R^{(2)}&=\int_{\R^n}\Big[-\tfrac12|v|^2v\cdot\nabla(\eta_R^2)
	-\tfrac12|w|^2v\cdot\nabla(\eta_R^2)
	+(v\cdot w)w\cdot\nabla(\eta_R^2)
	-Pv\cdot\nabla(\eta_R^2)
	\Big]\dd x.
\end{aligned}
\]
The term $I_R$ is uniformly bounded because
\begin{align}\label{1.16}
	|I_R|
	&\leq
	\frac{M_U}{2}
	\left(\norm{v}_{L^2}^2+\norm{w}_{L^2}^2\right)
	+
	M_B\norm{v}_{L^2}\norm{w}_{L^2}.
\end{align}
Since $\supp\nabla\eta_R\subset A_R$ and
$\|\nabla\eta_R\|_\infty\le C/R$, we have
\[
\left|E_R^{(1)}\right|
\le
\frac{C}{R}\left[
\|U\|_\infty(\|v\|_2^2+\|w\|_2^2)
+\|B\|_\infty\|v\|_2\|w\|_2
\right]
\longrightarrow0,
\]
and
\[
\left|E_R^{(2)}\right|\le \frac{C}{R}\int_{A_R} |v|^3\,\dd x+\frac{C}{R}\int_{A_R}|v||w|^2\,\dd x +\left|\int_{A_R}Pv\cdot\nabla\eta_R\,\dd x\right|.
\]
The condition $v\in L^2\cap L^{p,\infty}$, with $p>3$, implies
$v\in L^3$. Therefore
\begin{equation}\label{eq:u-weak}
\frac1R\int_{A_R}|v|^3\,\dd x
\le
\frac1R\|v\|_{L^3(A_R)}^3
\longrightarrow0.
\end{equation}
For the mixed term, choose $s$ such that
\[
\frac{2p}{p-1}<s<\min\left(4,q\right).
\]
Such $s$ exists because $1/p+2/q<1$ and $p>3$ give $q>2p/(p-1)$ and $4>2p/(p-1)$. Define $r$ by $1/r+2/s=1$. $s<4$ gives $r>2$; $s>2p/(p-1)$ gives $r<p$. The assumptions $v\in L^2\cap L^{p,\infty}$ and $w\in L^2\cap L^{q,\infty}$ then gives $v\in L^r$ and $w\in L^s$. Applying H\"older's inequality, we obtain
\begin{equation}
\frac{1}{R}\int_{A_R}|v||w|^2\,\dd x\leq \frac{1}{R}\|v\|_{L^r(A_R)}\|w\|_{L^s(A_R)}^2\longrightarrow 0.
\end{equation}
By Lemma~\ref{lem:pressure-noncompact}, write $P=Q+C$.
The constant contributes zero because
\[
\int_{\R^n}v\cdot\nabla\eta_R\,\dd x
=
-\int_{\R^n}\eta_R\diver v\,\dd x=0.
\]
For the linear pressure,
\begin{equation}
\left|
\int_{\R^n}Q_{\mathrm{lin}}v\cdot\nabla\eta_R\,\dd x
\right|
\le
\frac{C}{R}\|Q_{\mathrm{lin}}\|_2\|v\|_2
\longrightarrow0.
\end{equation}
Since $v\in L^3$, the Riesz-transform bound gives
$Q_{vv}\in L^{3/2}$, and hence
\begin{equation}
\left|
\int_{\R^n}Q_{vv}v\cdot\nabla\eta_R\,\dd x
\right|
\le
\frac{C}{R}
\|Q_{vv}\|_{L^{3/2}(A_R)}
\|v\|_{L^3(A_R)}
\longrightarrow0.
\end{equation}
Finally, using $w\in L^s$ and $v\in L^r$, we obtain
\begin{equation}\label{eq:Pw-weak}
\left|
\int_{\R^n}Q_{ww}v\cdot\nabla\eta_R\,\dd x
\right|
\le
\frac{C}{R}
\|Q_{ww}\|_{L^{s/2}(A_R)}
\|v\|_{L^r(A_R)}
\longrightarrow0.
\end{equation}
Thus $E_R^{(2)}\to0$. Combining the estimates of $I_R$, $E_R^{(1)}$ and $E_R^{(2)}$ and  \eqref{eq-energy-nocompact-cut}, we obtain
\[
\left|\ip{\Lambda_1^{2\alpha}v}{\eta_R^2v}\right|+\left|	\ip{\Lambda_1^{2\beta}w}{\eta_R^2w}\right|\le  \frac{M_U}{2}
\left(\norm{v}_{L^2}^2+\norm{w}_{L^2}^2\right)
+
M_B\norm{v}_{L^2}\norm{w}_{L^2}+o(1).
\]
Applying Lemma \ref{lem:localized} and Lemma \ref{lem:cutoff-error} for $0<\alpha,\beta<1$ and applying Lemma \ref{lem:classical-localized} for $\alpha$ or $\beta$ equals 1, we have
\[\norm{\Lambda_1^\alpha(\eta_Rv)}_{L^2}^2
+
\norm{\Lambda_1^\beta(\eta_Rw)}_{L^2}^2
\le\frac{M_U}{2}
\left(\norm{v}_{L^2}^2+\norm{w}_{L^2}^2\right)
+
M_B\norm{v}_{L^2}\norm{w}_{L^2}+o(1).
\]
Since \(\eta_Rv\to v\) and \(\eta_Rw\to w\) in \(L^2\), weak compactness gives
\begin{equation}\label{eq:directional-energy-entry}
\Lambda_1^\alpha v\in L^2,
\qquad
\Lambda_1^\beta w\in L^2.
\end{equation}

We then show the energy identity. Since $\Lambda_1^\alpha v$ and $\Lambda_1^\beta\in L^2(\R^n)$, by Lemma \ref{lem:directional-energy-cutoff}, the pairing in $D_R$ converge:
\[\ip{\Lambda_1^{2\alpha}v}{\eta_R^2v}\longrightarrow \|\Lambda_1^{\alpha} v\|_{L^2}^2,\quad 	\ip{\Lambda_1^{2\beta}w}{\eta_R^2w}\longrightarrow \|\Lambda_1^\beta w\|_{L^2}^2.
\]
The integrand defining $I(v,w)$ is integrable because
		$\nabla U,\nabla B\in L^\infty$ and $v,w\in L^2$.
		Dominated convergence gives $I_R\to I(v,w)$. Passing to the
		limit in \eqref{eq-energy-nocompact-cut} proves the proposition.
	\end{proof}
	
	\subsection{Proof of Theorem~\ref{thm:intro-noncompact}}
	
	By Remark~\ref{rem:concentration-attainment}, choose intervals
	$I_v,I_w$ of lengths $\ell_v,\ell_w$ such that
	\[
	\int_{\{x_1\in I_v\}}|v|^2\,\dd x
	\ge\kappa_v\|v\|_2^2,
	\qquad
	\int_{\{x_1\in I_w\}}|w|^2\,\dd x
	\ge\kappa_w\|w\|_2^2.
	\]
	Lemma~\ref{lem:directional-concentration} and \eqref{eq:directional-energy-entry} gives
	\begin{equation}\label{eq:concentration-v-w}
	\|v\|_2^2\le a\|\Lambda_1^\alpha v\|_2^2,
	\qquad
	\|w\|_2^2\le d\|\Lambda_1^\beta w\|_2^2.
	\end{equation}
	This includes the endpoint $s=1$, where
	$\|\Lambda_1f\|_2=\|\partial_1f\|_2$.
	
	Set
	\[
	X=\|\Lambda_1^\alpha v\|_2,
	\qquad
	Y=\|\Lambda_1^\beta w\|_2.
	\]
	By Proposition~\ref{prop:noncompact-energy},
	\[
	X^2+Y^2+I(v,w)=0.
	\]
	Moreover,
	\[
	|I(v,w)|
	\le
	\frac{M_U}{2}(\|v\|_2^2+\|w\|_2^2)
	+M_B\|v\|_2\|w\|_2.
	\]
	The concentration estimates imply
	\[
	\|v\|_2\|w\|_2\le\sqrt{ad}\,XY.
	\]
	It follows that
	\[
	(2-M_Ua)X^2+(2-M_Ud)Y^2
	-2M_B\sqrt{ad}\,XY\le0.
	\]
By \eqref{eq:noncompact}, we obtain $X=Y=0$.
	The concentration estimate \eqref{eq:concentration-v-w} then gives $v=w=0$, consequently, $\nabla P=0$.

	\bigskip
	
		{\bf Data availability}\quad No data are associated with this article.
	
	\section*{Declarations} 
	
	{\bf Conflict of interest}\quad The authors have no conflict of interest to declare that are relevant to the content of this article.

	\appendix
	
	\section{Proof of the fractional compact-support lemma}
	\label{compact-support}
	
	\begin{proof}[Proof of Lemma~\ref{lem:compact-support}]
		Choose $M>0$ such that $\supp f\subset[-M,M]$. For $t>M$,
		the singular-integral formula gives
		\[
		(-\partial_t^2)^sf(t)
		=
		-c_s\int_{\R}\frac{f(r)}{(t-r)^{1+2s}}\,\dd r.
		\]
		Since $(-\partial_t^2)^sf$ is compactly supported, this
		integral vanishes for all sufficiently large $t$.
		
		For such $t$ and $r\in[-M,M]$, expand
		\[
		(t-r)^{-1-2s}
		=
		t^{-1-2s}\sum_{k=0}^\infty
		\frac{(1+2s)_k}{k!}\left(\frac rt\right)^k.
		\]
		The series converges uniformly for $r\in[-M,M]$ when
		$t>M$. Consequently, the power series
		\[
		\sum_{k=0}^\infty
		\frac{(1+2s)_k}{k!}\,m_k z^k,
		\qquad
		m_k=\int_{\R}r^kf(r)\,\dd r,
		\]
		vanishes for all sufficiently small positive $z=1/t$.
		Uniqueness of power-series coefficients gives $m_k=0$ for
		every $k\ge0$.
		
		The Fourier transform of the compactly supported function $f$
		extends to an entire function, and
		\[
		\widehat f^{(k)}(0)=(-i)^km_k=0
		\qquad(k\ge0).
		\]
		Hence $\widehat f\equiv0$, and $f=0$.
	\end{proof}


\begin{thebibliography}{99}
		
		\bibitem{BWZ}
		U. Biccari, M. Warma, and E. Zuazua,
		\emph{Local elliptic regularity for the Dirichlet fractional Laplacian},
		Adv. Nonlinear Stud. \textbf{17}(2) (2017), 387--409.
		
		\bibitem{CF}
		J. A. Carrillo and L. C. F. Ferreira, 
		\emph{Self-similar solutions and large time asymptotics for the dissipative quasi-geostrophic equation}, 
		Monatsh.Math. {\bf151} (2007), 111--142
		
		
		\bibitem{Chae2014}
		D. Chae,
		\emph{Liouville-type theorems for the forced Euler equations and the Navier--Stokes equations},
		Commun. Math. Phys. \textbf{326} (2014), 37--48.
		
		\bibitem{Chae2020}
		D. Chae,
		\emph{Note on the Liouville type problem for the stationary Navier--Stokes equations in $\R^3$},
		J. Differential Equations \textbf{268} (2020), 1043--1049.
		
		\bibitem{Chae2025}
		D. Chae,
		\emph{On the Liouville type theorems for the stationary Navier--Stokes equations in $\R^3$},
		J. Differential Equations \textbf{445} (2025), 113597.
		
		\bibitem{Chae2026}
		D. Chae,
		\emph{Liouville Type Theorems for the Stationary Navier--Stokes Equations in $\R^3$},
		Commun. Math. Phys. \textbf{407} (2026), article~53.
		
		\bibitem{ChaeAnisotropicMHD2023}
		D. Chae,
		\emph{Anisotropic Liouville type theorem for the MHD system in $\R^n$},
		J. Math. Phys. \textbf{64} (2023), 121501.
		
		\bibitem{ChaeAnisotropicNS2023}
		D. Chae,
		\emph{Anisotropic Liouville type theorem for the stationary Navier--Stokes equations in $\R^3$},
		Appl. Math. Lett. \textbf{142} (2023), 108655.
		
		\bibitem{ChaeLee2024}
		D. Chae and J. Lee,
		\emph{On Liouville type results for the stationary MHD in $\R^3$},
		Nonlinearity \textbf{37} (2024), 095006.
		
		\bibitem{ChaeWolf2016}
		D. Chae and J. Wolf,
		\emph{On Liouville type theorems for the steady Navier--Stokes equations in $\R^3$},
		J. Differential Equations \textbf{261} (2016), 5541--5560.
		
		\bibitem{ChaeWolf2019}
		D. Chae and J. Wolf,
		\emph{On Liouville type theorem for the stationary Navier--Stokes equations},
		Calc. Var. Partial Differential Equations \textbf{58} (2019), article~111.
		
		\bibitem{ChaeWolfMHD2021}
		D. Chae and J. Wolf,
		\emph{On Liouville type theorems for the stationary MHD and Hall--MHD systems},
		J. Differential Equations \textbf{295} (2021), 233--248.
		
		\bibitem{ChamorroJarrinLemarie2021}
		D. Chamorro, O. Jarr\'in, and P.-G. Lemari\'e-Rieusset,
		\emph{Some Liouville theorems for stationary Navier--Stokes equations in Lebesgue and Morrey spaces},
		Ann. Inst. H. Poincar\'e C Anal. Non Lin\'eaire \textbf{38} (2021), no.~3, 689--710.
		
		\bibitem{ChamorroPoggi2025}
		D. Chamorro and B. Poggi,
		\emph{On an almost sharp Liouville-type theorem for fractional Navier--Stokes equations},
		Publ. Mat. \textbf{69} (2025), 27--43.
		
		
		
		\bibitem{ChenLiWang2022}
		X. Chen, S. Li, and W. Wang,
		\emph{Remarks on Liouville-type theorems for the steady MHD and Hall-MHD equations},
		J. Nonlinear Sci. \textbf{32} (2022), article~12.
		
		\bibitem{ChoNeustupaYang2024}
		Y. Cho, J. Neustupa and M. Yang,
		\emph{New Liouville type theorems for the stationary Navier--Stokes, MHD, and Hall--MHD equations},
		Nonlinearity \textbf{37} (2024), 035007.
		
		\bibitem{DNPV}
		E. Di Nezza, G. Palatucci and E. Valdinoci,
		\emph{Hitchhiker's guide to the fractional Sobolev spaces},
		Bull. Sci. Math. \textbf{136} (2012), no. 5, 521--573.
		
		\bibitem{DMZ}
		Y. Dong and Z. Zhang,
		\emph{Liouville Type Theorems for the $3$D Steady Navier--Stokes Equations Under Gradient Conditions},
		Appl. Math. \& Optim. \textbf{94} (2026), 59.
		
		\bibitem{EH}
		T. Eguchi and G. Vergara-Hermosilla.
		\emph{ Liouville theorems for fractional Navier-Stokes equations via the radial velocity component}.
		Bull. Sci. Math., {\bf213} (2026), 103953.
		
		\bibitem{Galdi}
		G.~P. Galdi,
		\emph{An Introduction to the Mathematical Theory of the Navier--Stokes Equations: Steady-State Problems},
		2nd ed., Springer, New York, 2011.
		
		\bibitem{GSU}
		T. Ghosh, M. Salo and G. Uhlmann,
		\emph{The Calder\'on problem for the fractional Schr\"odinger equation},
		Anal. PDE \textbf{13} (2020), no. 2, 455--475.
		
		
		\bibitem{Jarrin2020}
		O. Jarr\'in,
		\emph{A remark on the Liouville problem for stationary Navier--Stokes equations in Lorentz and Morrey spaces},
		J. Math. Anal. Appl. \textbf{486} (2020), no.~1, 123871.
		
		\bibitem{JarrinVergara2024}
		O. Jarr\'in and G. Vergara-Hermosilla,
		\emph{An $L^p$-theory for fractional stationary Navier--Stokes equations},
		J. Elliptic Parabol. Equ. \textbf{10} (2024), 859--898.
		
		\bibitem{KozonoTerasawaWakasugi2017}
		H. Kozono, Y. Terasawa and Y. Wakasugi,
		\emph{A remark on Liouville-type theorems for the stationary Navier--Stokes equations in three space dimensions},
		J. Funct. Anal. \textbf{272} (2017), no.~2, 804--818.
		
		\bibitem{Kwa}
		M. Kwa\'snicki,
		\emph{Ten equivalent definitions of the fractional Laplace operator},
		Fract. Calc. Appl. Anal. \textbf{20} (2017), no. 1, 7--51.
		
		\bibitem{LeeLee2026}
		J. Lee and J. Lee,
		\emph{Liouville-type theorems for the stationary fractional Navier--Stokes equations in $\R^n$},
		arXiv:2606.28959v1, 2026.
		\url{https://arxiv.org/abs/2606.28959}
		
		\bibitem{ON}
		R. O'Neil,
		\emph{Convolution operators and $L(p,q)$ spaces},
		Duke Math. J. \textbf{30} (1963), 129--142.
		
		\bibitem{ROS}
		X. Ros-Oton and J. Serra,
		\textit{The Dirichlet problem for the fractional Laplacian: Regularity up to the boundary},
		J. Math. Pures Appl. \textbf{101}(3) (2014), 275--302.
		
		
		\bibitem{Seregin2016}
		G. Seregin,
		\emph{Liouville type theorem for stationary Navier--Stokes equations},
		Nonlinearity \textbf{29} (2016), 2191--2195.
		
		\bibitem{Seregin2019}
		G. Seregin,
		\emph{Remarks on Liouville type theorems for steady-state Navier--Stokes equations},
		St. Petersburg Math. J. \textbf{30} (2019), no.~2, 321--328.
		
		\bibitem{SereginWang2020}
		G. Seregin and W. Wang,
		\emph{Sufficient conditions on Liouville type theorems for the 3D steady Navier--Stokes equations},
		St. Petersburg Math. J. \textbf{31} (2020), 387--393.
		
		\bibitem{Stein}
		E. M. Stein,
		\emph{Singular Integrals and Differentiability Properties of Functions},
		Princeton University Press, Princeton, NJ, 1970.
		
		\bibitem{Tsai2021}
		T.-P. Tsai,
		\emph{Liouville type theorems for stationary Navier--Stokes equations},
		Partial Differ. Equ. Appl. \textbf{2} (2021), article~10.
		
		\bibitem{WG}
		P. Wang and Z. Guo,
		\emph{Some Liouville-Type Results for the 3D Incompressible MHD Equations},
		J. Geom. Anal., \textbf{34} (2024), article~39.
		
		
		\bibitem{WangYang2025}
		W. Wang and G. Yang,
		\emph{Liouville type theorems for the 3D stationary MHD and Hall-MHD equations with non-zero constant vectors at infinity},
		Nonlinearity \textbf{38} (2025), 065007.
		
		\bibitem{WangYangYu2025}
		W. Wang, G. Yang and J. Yu,
		\emph{Liouville type theorems for the fractional Navier--Stokes equations without the integrability condition of velocity in $\R^3$},
		arXiv:2505.04895, 2025.
		\url{https://arxiv.org/abs/2505.04895}
		
		\bibitem{WangXiao2018}
		Y. Wang and J. Xiao,
		\emph{A Liouville problem for the stationary fractional Navier--Stokes--Poisson system},
		J. Math. Fluid Mech. \textbf{20} (2018), 485--498.
		
		\bibitem{Yang2022}
		J. Yang,
		\emph{On Liouville type theorem for the steady fractional Navier--Stokes equations in $\R^3$},
		J. Math. Fluid Mech. \textbf{24} (2022), article~81.
		
		
		\bibitem{Zeng2025}
		Y. Zeng,
		\emph{On Liouville-type theorems for the 3D stationary fractional MHD system in anisotropic Lebesgue spaces},
		Z. Angew. Math. Phys. \textbf{76} (2025), article~181.
		
		\bibitem{Zhang}
		Z. Zhang,
		\emph{New Liouville type theorems for 3D steady incompressible MHD equations and Hall-MHD equations},
		arXiv:2503.13202v3, 2025.
		\url{https://doi.org/10.48550/arXiv.2503.13202}
		
	\end{thebibliography}
\end{document}